\documentclass[a4paper,twoside,11pt]{amsart}
\usepackage[a4paper,left=3cm,right=2.5cm, top=3cm, bottom=3cm]{geometry}

\usepackage{algorithm}
\usepackage{algpseudocode}
\usepackage{amsaddr}
\usepackage{amssymb}
\usepackage{bbm}
\usepackage{caption}
\usepackage{float}
\usepackage{graphicx}
\usepackage{geometry} 
\usepackage{mathtools}
\usepackage{setspace}
\usepackage{subcaption}
\usepackage{tabu}
\usepackage[table]{xcolor}
\usepackage{nicefrac}

\graphicspath{{Figures/}}

\newcommand{\cC}{\mathcal{C}}
\newcommand{\chU}{\hat{\cU}}
\newcommand{\cL}{\mathcal{L}}
\newcommand{\cO}{\mathcal{O}}
\newcommand{\cP}{\mathcal{P}}
\newcommand{\cQ}{\mathcal{Q}}
\newcommand{\cS}{\mathcal{S}}
\newcommand{\cT}{\mathcal{T}}
\newcommand{\cU}{\mathcal{U}}
\newcommand{\cX}{\mathcal{X}}

\newcommand{\dd}{\mathrm{d}}
\newcommand{\dist}{\mathrm{dist}}
\newcommand{\e}{E}
\newcommand{\Fr}{\mathfrak{F}}
\newcommand{\Ft}{\mathfrak{F}_t}
\newcommand{\he}{\hat{\e}}

\newcommand{\hU}{\hat{U}}
\newcommand{\im}{i}

\newcommand{\mR}{\mathbb{R}}

\newcommand{\tF}{\tilde{F}}
\newcommand{\tomega}{\tilde{\omega}}
\newcommand{\tshift}{(\vhr\cdot\vx)/c_0}

\newcommand{\tU}{\tilde{U}}
\newcommand{\tz}{\tilde{z}}
\newcommand{\tZ}{\tilde{Z}}

\newcommand{\vhr}{\hat{\vct{r}}}

\newcommand{\vs}{\vct{s}}
\newcommand{\vx}{\vct{x}}
\newcommand{\vy}{\vct{y}}

\newcommand{\vzero}{\vct{0}}

\newcommand{\vct}[1]{\mathbf{#1}}

\newtheorem{thm}{Theorem}[section]
\newtheorem{lem}[thm]{Lemma}

\newtheorem{rem}[thm]{Remark}

\newtheorem{dfn}[thm]{Definition}

\definecolor{bl}{rgb}{0 0 0}

\definecolor{ilariablue}{rgb}{0 0.4 0.85}
\definecolor{ilariared}{rgb}{0.95 0.1 0}

\definecolor{sjoerdgreen}{rgb}{0 0.7 0.2}

\definecolor{dmitrycol}{rgb}{0.3 0.1 0.8}

\title[An Adaptive FEM for scattering problems]{An adaptive finite element method for high-frequency scattering problems with smoothly varying coefficients} 
\author{Anton Arnold$^{1*}$, Sjoerd Geevers$^{2*}$, Ilaria
  Perugia$^{2*}$, Dmitry Ponomarev $^{1,3*}$}
\address{{\small
$^1$ Institute for Analysis and Scientific Computing, Vienna
University of Technology\\
Wiedner Hauptstrasse 8-10, 1040 Vienna, Austria\\
$^2$ Faculty of Mathematics, University of Vienna\\
Oskar-Morgenstern-Platz 1, 1090 Vienna, Austria\\
$^3$ St. Petersburg Department of V. A. Steklov Mathematical Institute, RAS,\\
Fontanka 27, 191023 St. Petersburg, Russia} 
}
\thanks{*A. Arnold, S. Geevers, and I. Perugia have been funded by the
Austrian Science Fund (FWF)
through the project F~65 ``Taming Complexity in Partial Differential
Systems''. I. Perugia has also been funded by the FWF through the
project P~29197-N32. A. Arnold and D. Ponomarev were supported by the bi-national FWF-project I3538-N32}

\begin{document}

\maketitle

\centerline{\today}

\begin{abstract}
We introduce a new numerical method for solving
time-harmonic acoustic scattering problems.
The main focus is on plane waves scattered by smoothly varying material inhomogeneities. The proposed method works for any frequency $\omega$, but is especially efficient for high-frequency problems. It is based on a time-domain approach and consists of three steps:
\emph{i)} computation of a suitable incoming plane wavelet with compact support in the propagation direction; 
\emph{ii)} solving a scattering problem in the time domain for the incoming plane wavelet;
\emph{iii)} reconstruction of the time-harmonic solution from the time-domain solution via a Fourier transform in time.
An essential ingredient of the new method is a front-tracking mesh adaptation algorithm for solving the problem in \emph{ii)}. By exploiting the limited support of the wave front, this allows us to make the number of the required degrees of freedom to reach a given accuracy significantly less dependent on the frequency $\omega$. 
We also present a new algorithm for computing the Fourier transform in \emph{iii)} that exploits the reduced number of degrees of freedom corresponding to the adapted meshes. 
Numerical examples demonstrate the advantages of the proposed method and the fact
that the method can also be applied with external source terms such as point sources and sound-soft scatterers. The gained efficiency, however, is limited in the presence of trapping modes.
\end{abstract}
\medskip

{\footnotesize
\noindent
{\bf Keywords} {Helmholtz equation,
    scattering problem,
    time-domain wave problem,
    limiting amplitude principle,
    adaptive FEM,
    front-tracking mesh
  }\\[0.1cm]
\noindent
{\bf Mathematics Subject Classification } {35J05, 
35L05, 
65M60, 
65M50 
}}
\medskip

\section{Introduction}
We consider time-harmonic wave scattering problems
in inhomogeneous media with smooth\-ly varying
material properties.
Such
problems become notoriously hard to solve when the angular frequency
$\omega$ is large.
To obtain a given accuracy with a standard finite
difference or finite element method requires at least $\cO(\omega^d)$
degrees of freedom~\cite{babuska1997}, where $d$ denotes the number of space dimensions. On top of that, standard iterative solvers and multigrid methods
break down or converge slowly for high frequencies \cite{ernst2012}.
 
One way to reduce the computational complexity for large frequencies
is by combining finite element methods with asymptotic methods
\cite{giladi2001, nguyen2015}.
An asymptotic method, such as the geometrical optics method, is used to determine the wave propagation directions and a plane-wave finite element method is then used to solve the scattering problem. A drawback of this approach is that standard geometrical optics does not account for diffracted fields and incorporating diffraction phenomena typically requires an \textit{ad hoc} approach. 
Instead of using an asymptotic method, one can also extract the dominant wave propagation directions from the solution of
a lower-frequency scattering problem \cite{fang2017}.

Another way to reduce the computational complexity is by using a
classical finite difference or finite element method with a standard
iterative solver, but in combination with a sweeping preconditioner
\cite{engquist2011a, engquist2011b, stolk2016}; for a more recent overview of several sweeping preconditioning methods, see \cite{gander2019}. The number of iterations then becomes nearly independent of the frequency, which means that the computational complexity scales almost as $\cO(\omega^d)$. 

Instead of solving the scattering problem directly in the frequency
domain, one can also solve the scattering problem in the time
domain. The time-harmonic solution can be obtained from a solution to
a time-dependent wave equation by exploiting the limiting amplitude
principle \cite{ladyzhenskaya1957,vainberg1966,morawetz1962} or by
applying a Fourier transformation in time.
Classical time domain methods are the finite difference and
finite element time domain methods \cite{yee1966, taflove2005}.
Time-domain methods that are specifically devised for solving frequency-domain problems include the controllability
method~\cite{bristeau1998, glowinski2006}, with its spectral version~\cite{heikkola2007a} and
its extensions~\cite{grote2019,grote2020}, the WaveHoltz
method~\cite{appelo2020}, and the time-domain preconditioner 
of~\cite{stolk2020}.

While both frequency-domain and time-domain methods are commonly used
in practice and are expected to remain relevant in the future, this
paper will focus on the time-domain approach. Some of the advantages
of time-domain methods are that they are inherently parallel and
straightforward to implement, without the need of storing Krylov
subspaces and matrix factorisations and without the need of
implementing linear solvers and moving absorbing boundary
layers. However, for classical finite difference and finite element
time domain methods, the number of degrees of freedom is at least
$\cO(\omega^d)$ and the number of time steps is at least $\cO(\omega)$
due to the CFL condition, resulting in a computational complexity of
at least $\cO(\omega^{d+1})$. In this paper, we present an adaptive
finite element time-domain method that reduces the average number of
degrees of freedom per time step to almost $\cO(\omega^{d-1})$, resulting in a computational cost that scales almost as $\cO(\omega^d)$.

The main idea of 
adaptive finite element time-domain methods is to automatically
adapt the mesh over time in such a way that fine elements are used
near the wave front and coarser elements are used away from the wave
front. The mesh adaptation algorithms are typically driven by
\textit{a posteriori} error estimators/indicators. Adaptive finite
element methods for the wave equation were studied for a conforming finite
element discretisation in space combined with a discontinuous Galerkin
discretisation in time \cite{johnson1993, li1998, thompson2005}, the
Crank--Nicholson scheme in time
\cite{bangerth2004,gorynina2019,gorynina2019b}, the implicit Euler
scheme in time \cite{bernardi2005, georgoulis2013}, and the leap-frog
scheme in time \cite{georgoulis2016}. Adaptive finite element methods
based on a discontinuous Galerkin discretisation in space were
presented and analysed in
\cite{adjerid2002,adjerid2002b,adjerid2010}. An anisotropic adaptive
mesh refinement algorithm was studied in \cite{picasso2010}. Adaptive
finite element schemes tailored for a given goal functional were
studied in \cite{bangerth1999,bangerth2010}. 
Adaptive finite element schemes have also been studied for other time-dependent problems such as the Stefan problem~\cite{nochetto1991} and nonlinear wave equations~\cite{alauzet2003}.
Here, we present a new adaptive finite element time-domain method that is tailored for efficiently solving time-harmonic scattering problems and has the following distinguishing features:
\begin{itemize}
  \item The source term is obtained from an incoming plane wavelet with compact support in the direction of propagation of width $\cO(\omega^{-1})$.
  \item The adapted meshes are obtained from a set of nested meshes that are defined \textit{a priori}.
  \item The mesh is not updated at each time step, but only after every $m$ time steps. In the numerical examples, we have $m\sim 10-100$.
  \item The time-harmonic field is obtained using an adapted algorithm for computing the Fourier transform in time that exploits the reduced number of degrees of freedom of the adapted meshes.
\end{itemize}
In the numerical section, we present an implementation of the method using a fully explicit conforming finite element time-stepping scheme combined with the perfectly matched layer of \cite{grote2010}.

The paper is organised as follows: In Section~\ref{sec:timedomain}, we
explain how we solve the time-harmonic scattering problem using a
time-domain approach. The adaptive finite element method for solving the time-dependent problem and the adapted method for computing the Fourier transform in time are then given in Section~\ref{sec:AFEM}. Details of the numerical implementation and several numerical examples are given in Section~\ref{sec:numerics}. Finally, our findings are summarised in Section~\ref{sec:conclusion}.

\section{Solving the time-harmonic scattering problem in the time domain}
\label{sec:timedomain}
We are interested in solving the time-harmonic wave scattering problem in~$d$ dimensions governed by the Helmholtz equation 
\begin{subequations}
\label{eq:Helmholtz1}
\begin{align}
-\omega^2(U_S+U_I) - \beta^{-1}\nabla\cdot (\alpha\nabla (U_S+U_I)) &= 0 &&\text{in }\mR^d, \\
[{\text{far field radiation condition on }U_S}], &  &&
\end{align}
\end{subequations}%
where $U_S=U_S(\vx)$ is the scattered wave field that needs to be resolved, $U_I=U_I(\vx)$ is a given incoming plane wave, $\omega>0$ is the angular frequency, $\nabla$ and $\nabla\cdot$ denote the gradient and divergence operator, respectively, and $\alpha=\alpha(\vx)\geq\alpha_{\min}>0$ and $\beta=\beta(\vx)\geq\beta_{\min}>0$ are two material parameters that are assumed to vary smoothly in space. Our main interest is the case where $\omega$ is large. 
We assume that there exists a bounded domain $\Omega_{in}\subset\mR^d$ such that $\alpha$ and $\beta$ are constant, say $\alpha=\alpha_0$ and $\beta=\beta_0$, in the exterior domain $\Omega_{ex}:=\mR^d\setminus\overline{\Omega}_{in}$. 
We also assume that the incoming plane wave is of the form
$U_I(\vx)=e^{\im\omega\tshift}$, with $\im$ the imaginary unit
($\im^2=-1$), $\vhr$ a unit direction vector, and $c_0:=\sqrt{\alpha_0/\beta_0}$ the wave propagation speed in the exterior domain.

We can rewrite \eqref{eq:Helmholtz1} as
\begin{subequations}
\label{eq:Helmholtz2}
\begin{align}
-\omega^2U_S - \beta^{-1} \nabla\cdot (\alpha\nabla U_S) &= F  &&\text{in }\mR^d, \\
[\text{far field radiation condition on }U_S], &  &&
\end{align}
\end{subequations}
with $F:=\omega^2U_I+\beta^{-1}\nabla\cdot (\alpha\nabla U_I)$. 
Note that $F=0$ in $\Omega_{ex}$.

\begin{rem}
In acoustic scattering, $U_S$ is the scattered pressure field, $\alpha$ is the reciprocal of the mass density $\rho=\rho(\vx)$ of the medium, 
and $\beta$ is the reciprocal of $\rho c^2$, with $c=c(\vx)$ being the wave propagation speed of the medium. 
\end{rem}

A common way to obtain the scattered wave field $U_S$ is by solving a
wave scattering problem in the time domain for a time-harmonic source
term of the form $F(\vx)e^{-\im\omega t}$. 
If the \emph{limiting amplitude principle} is valid \cite{odeh1961,eidus1969},
the time-dependent scattered wave field converges to
$U_S(\vx)e^{-\im\omega t}$ as $t$ tends to infinity; see Appendix~\ref{sec:appFT}. 

Alternatively, we can compute the scattered wave field $u_S(\vx,t)$ for a suitable source term $f(\vx,t)$ with compact support in
space and time. Let $\Ft$ denote the Fourier transform with respect to time,
namely $\Ft[\varphi](\omega')=\int_{\mathbb R} e^{-\im\omega' t}\varphi(t)\,\dd t$. 
If $\Ft[f](\cdot,-\omega)=F$, then it follows from the limiting
amplitude principle that $U_S=\Ft[u_S](\cdot,-\omega)$; see
Lemma~\ref{lem:FT} in Appendix~\ref{sec:appFT}. Our proposed numerical method is based on this latter approach. In particular, we solve the scattered wave field $u_S(\vx,t)$ corresponding to a single incoming plane wavelet $u_I(\vx,t)$, defined such that $U_I=\Ft[u_I](\cdot,-\omega)$, and then compute $U_S=\Ft[u_S](\cdot,-\omega)$. By \textit{plane wavelet} we mean a plane wave with compact support in the propagation direction. We describe this approach in detail in the three steps below.

\emph{Step 1.~Defining the incoming plane wavelet.}
We consider an incoming plane wavelet of the form $u_I(\vx,t):=\omega\psi(\omega(t-\tshift))$,
where $\psi=\psi(\xi)$ is some real-valued, smooth function with $\text{supp}(\psi)=[-\xi_0,\xi_0]$, where $\xi_0>0$ is some constant independent of $\omega$, and such that its Fourier transform $\Fr[\psi]$ satisfies $\Fr[\psi](-1)=\int_{-\xi_0}^{\xi_0} e^{\im \xi}\psi(\xi) \;\dd \xi =1$. 
The incoming wave field is thus a traveling plane wavelet of amplitude $\cO(\omega)$ and with a support of width $2c_0\xi_0\omega^{-1}=\cO(\omega^{-1})$. 
The Fourier transform of $u_I$ is given by
\begin{align*}
\Ft[u_I](\vx,\tomega) &= \Ft[\omega\psi(\omega(t-\tshift))] (\vx,\tomega)  \\
&= \omega \Ft[\psi(\omega(t-\tshift))] (\vx,\tomega) \\
&= \omega e^{-\im\tomega\tshift} \Ft[\psi(\omega t)](\vx,\tomega) \\
&= e^{-\im\tomega\tshift} \Ft[\psi(t)] (\vx,{\tomega}/{\omega}) \\
&= e^{-\im\tomega\tshift} \Fr[\psi]({\tomega}/{\omega})
\end{align*}
and since $\Fr[\psi](-1)=1$, we therefore have $\Ft[u_I](\cdot,-\omega)=U_I$. 
An illustration of $u_{I}$ is given in Figure \ref{fig:uIn}. 

\medskip

\emph{Step 2.~Solving a wave scattering problem in the time-domain.}
Having defined the incoming plane wavelet $u_I$, we next solve the scattered wave field $u_S(\vx,t)$ given by the wave equation
\begin{subequations}
\label{eq:wave_S1}
\begin{align}
\partial_t^2(u_S+u_I) - \beta^{-1}\nabla\cdot (\alpha\nabla (u_S+u_I)) &=0 &&\text{in }\mR^d\times(t_0,\infty), \\
[\text{zero initial conditions on $u_S$ at $t=t_0$}], & &&
\end{align}
\end{subequations}%
where $\partial_t^2$ denotes the second-order time derivative, and $t_0:=\inf_{\vx\in\Omega_{in}}\tshift-\xi_0 \omega^{-1}$ is the time when the incoming plane wavelet first enters $\Omega_{in}$. We can rewrite this equation as
\begin{subequations}
\label{eq:wave_S2}
\begin{align}
\partial_t^2u_S - \beta^{-1}\nabla\cdot (\alpha\nabla u_S) &= f &&\text{in }\mR^d\times(t_0,\infty), \label{eq:wave_S1a} \\
[\text{zero initial conditions on $u_S$ at $t=t_0$}], & &&
\end{align}
\end{subequations}%
where $f:= - \partial_t^2u_I + \beta^{-1}\nabla\cdot (\alpha\nabla u_I)$. Note that $f$ has only support in $\Omega_{in}\times(t_0,t_f)$, where $t_f:=\sup_{\vx\in\Omega_{in}} \tshift + \xi_0 \omega^{-1}$ is the time when the incoming plane wavelet leaves $\Omega_{in}$. Furthermore, since $\Ft[u_I](\cdot,-\omega)=U_I$, it follows that $\Ft[f](\cdot,-\omega)=F$.

\medskip

{\emph{Step 3.~Reconstructing the Helmholtz solution using a Fourier transform.}}
Since $\Ft[f](\cdot,-\omega)=F$, it follows that, after extending $u_S$ by zero in $\mR^d\times(-\infty,t_0)$, we have $\Ft[u_S](\cdot,-\omega)=U_S$.
Having computed the scattered wave field $u_S$ given by \eqref{eq:wave_S2}, we can thus reconstruct the solution to 
the Helmholtz equation given in~\eqref{eq:Helmholtz2} by computing the limit
\begin{align}
\label{eq:U_S}
U_S &=\Ft[u_S](\cdot,-\omega) = \lim_{t\rightarrow\infty}  \int_{t_0}^{t} e^{\im\omega \tau}u_S(\cdot,\tau) \;\dd\tau.
\end{align}

\medskip

\begin{figure}[h]
\begin{center}
\includegraphics[width=0.32\textwidth]{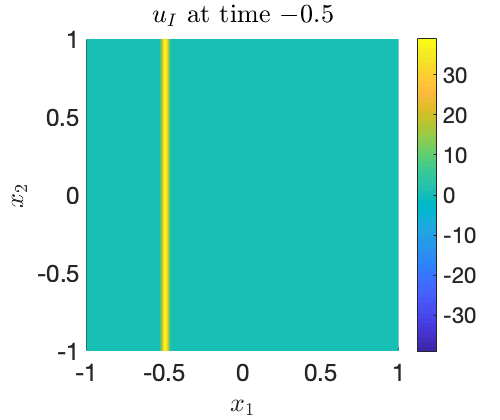}
\includegraphics[width=0.32\textwidth]{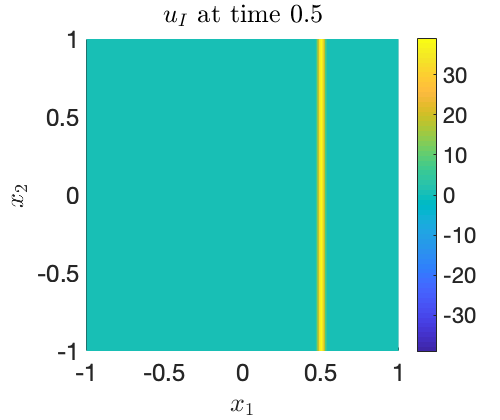}
\includegraphics[width=0.32\textwidth]{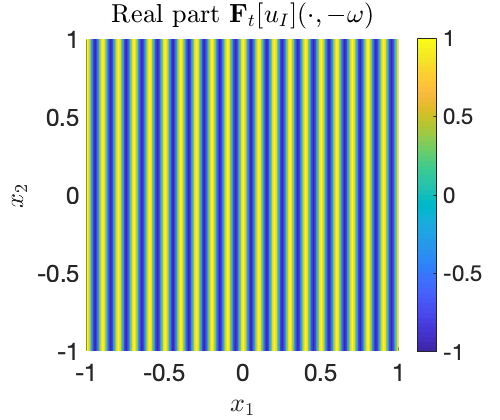}
\end{center}
\caption{Illustration of the incoming plane wavelet $u_I(\vx,t)=\omega\psi(\omega(t-\tshift))$ at time $t=-0.5$ (left) and $t=0.5$ (middle), and its Fourier transform $\Ft[u_I](\cdot,-\omega)$ (right). We chose $\omega=20\pi$, $\vhr=(1,0)$, $c_0=1$, and $\psi$ defined as in Section \ref{sec:num1D}.}
\label{fig:uIn}
\end{figure}

The advantage of this 
approach is that the scattered wave field $u_S$ corresponds to an incoming plane wavelet of amplitude $\cO(\omega)$ and with a support of width $\cO(\omega^{-1})$. 
Motivated by geometric optics, we expect that, when $\omega$ is large and when there are no
trapping modes, the solution $u_S$ is 
a travelling wave that has a steep gradient near the wave front, and a small gradient that is more or less independent of $\omega$ everywhere else; see also Figure \ref{fig:uh} in Section~\ref{sec:numerics}.
In the finite element approximation of $u_S$, we can exploit this property and significantly reduce the computational cost by using an adaptive, time-dependent spatial mesh, where a fine mesh is used near the wave front and a coarser mesh, with mesh width independent of~$\omega$, is used elsewhere.

\begin{rem}
We can readily extend the approach for a wave scattering problem that includes a sound-soft scatterer $\Omega_{sc}$. The wave scattering problem is then given by equation \eqref{eq:Helmholtz1}, but with a spatial domain $\mR^d\setminus\Omega_{sc}$ instead of $\mR^d$, and with an additional boundary condition of the form
\begin{align*}
U_S + U_I &= 0 &&\text{on }\partial\Omega_{sc}.
\end{align*}
The approach remains identical, except that in equations \eqref{eq:wave_S1} and \eqref{eq:wave_S2} in Step 2, the spatial domain is $\mR^d\setminus\Omega_{sc}$ instead of $\mR^d$ and the additional boundary condition is of the form
\begin{align*}
u_S &= -u_I && \text{on }(t_0,\infty)\times \partial\Omega_{sc}.
\end{align*}
\end{rem}

\begin{rem}
With a slight modification, the approach can also be applied to a wave scattering problem of the form in \eqref{eq:Helmholtz2} with $F=F(\vx)$ an arbitrary external source term with bounded support. In Step 1, we then only need to define $\psi$, but not $u_I$. In Step 2, we then solve equation \eqref{eq:wave_S2} for $f(\vx,t)=\omega\psi(\omega t)F(\vx)$, with $t_0:=-\xi_0 \omega^{-1}$ the earliest time when $f(\cdot,t)$ is non-zero. Note that $f$ only has support in $\text{supp}(F)\times(-\xi_0\omega^{-1}, \xi_0\omega^{-1})$. Step 3 remains unaltered. If $F$ is a source term with very local support in space, such as a point source, we again expect that we can solve the time-domain problem efficiently using an adaptive mesh.
\end{rem}

A description of the adaptive finite element method is given in the following section.

\section{An adaptive finite element method}
\label{sec:AFEM}
We aim to solve 
the wave equation given in \eqref{eq:wave_S2} by using a finite
element method with an adapted spatial mesh that is constantly updated
over time. We simultaneously update the right-hand side of
\eqref{eq:U_S} 
by applying a discretised Fourier transform in time that exploits the reduced number of degrees of freedom of the adapted spatial meshes.

The main idea of the adaptive finite element method is as follows: we consider a finite computational domain $\Omega\supset\Omega_{in}$ with an absorbing boundary and split the time domain
into small intervals $(T_{j-1},T_j)$ of length $\cO(\omega^{-1})$.
Let $u_h$ denote the finite element approximation to $u_S$.
At the beginning of each time interval, we construct an adapted mesh $\cT_j$ of the domain $\Omega$ based on the current discrete approximation $u_h(\cdot,T_{j-1})$. 
We then project $u_h(\cdot,T_{j-1})$ 
into the finite element space of 
$\cT_j$ and solve the discretised wave equation on $\cT_j$ for the time interval $(T_{j-1},T_j)$.

For constructing $\cT_j$, we aim for the coarsest possible mesh on
which $u_S$ can still be approximated accurately during the time interval $(T_{j-1},T_j)$. To construct such a mesh, we assume that the current discrete approximation $u_h(\cdot,T_{j-1})$ is an accurate approximation of $u_S(\cdot,T_{j-1})$ and take into account that $u_S$ travels during the time interval $(T_{j-1},T_j)$ and is generated by the incoming wave $u_I$. Let $c_{\max}:=\sup_{\vx\in\Omega}\sqrt{\alpha(\vx)/\beta(\vx)}$ denote the maximum wave propagation speed. The construction of $\cT_j$ 
consists of the following steps:  
\begin{itemize}
  \item \emph{Coarsen}. Coarsen the mesh $\cT_{j-1}$ and project $u_h(\cdot,T_{j-1})$ onto the discrete
    space associated with the coarser mesh.
  \item \emph{Estimate}. Compute the projection error, i.e. compute the difference between $u_h(\cdot,T_{j-1})$ and its projection onto the coarser mesh.
  \item \emph{Mark 1}. Mark all elements of the coarser mesh where the projection error is above a certain threshold. Also mark all elements of the coarser mesh that overlap with the support of the incoming wave $u_I(\cdot,T_{j-1})$.
  \item \emph{Mark 2}. Mark all elements of the coarser mesh that are within a distance $c_{\max}(T_j-T_{j-1})$ of elements that were marked in the first round. 
  \item \emph{Refine}. Refine the coarser mesh at all marked elements to obtain $\cT_j$.
\end{itemize}

To accurately approximate the Fourier transform of $u_h$ in time, we 
normally need to sample $u_h$ at all the degrees of freedom of a uniformly fine mesh at each time step.
By computing the Fourier transform using an adapted algorithm, we can significantly reduce the average number of sampling points per time step.

A detailed description of the complete method is given in the following subsections.

\subsection{A finite element method with a time-dependent mesh}
\label{sec:AFEM1}
Consider the wave scattering problem given in \eqref{eq:wave_S2}. To
approximate the scattered wave field $u_S$ using a finite element
method, we consider a bounded polygonal domain
$\Omega\supset\Omega_{in}$, and impose an absorbing boundary condition
or add an absorbing boundary layer at $\partial\Omega$. Since, however, the main steps of our adaptive finite element method do not depend on the type of boundary condition, we consider in this section 
a zero Dirichlet boundary condition $u_S|_{\partial\Omega}=0$ in order to simplify the presentation. We thus consider a wave equation of the form
\begin{subequations}
\label{eq:wave_S3}
\begin{align}
\partial_t^2 u - \beta^{-1}\nabla\cdot(\alpha\nabla u) &= f &&\text{in }\Omega\times (t_0,\infty), \\
u(\cdot,t_0) = \partial_tu(\cdot,t_0) &= 0 &&\text{in }\Omega, \\
u&=0 &&\text{on }\partial\Omega\times (t_0,\infty).
\end{align}
\end{subequations}

Let $T_{up} = \cO(\omega^{-1})$ be the time after which the mesh is updated and define $T_j:=t_0+jT_{up}$. 
Also, let $\cT_j$ denote the simplicial/square/cubic mesh of $\Omega$ used during the time interval $(T_{j-1},T_j]$. For any mesh $\cT$, let $\cU_{\cT}$ denote the corresponding finite element space, given by
\begin{align*}
\cU_{\cT} &:= \{u\in H_0^1(\Omega) \; |\; u\circ\phi_\e \in \chU \text{ for all }\e\in\cT \},
\end{align*} 
with $\phi_\e:\he\rightarrow \e$ the affine element mapping, $\he$ the reference element, and $\chU$ the polynomial reference space. Also, let $\cL_\cT:\cU_\cT\rightarrow\cU_\cT$ denote the discretisation of the spatial operator $u\mapsto-\beta^{-1}\nabla\cdot(\alpha\nabla u)$ for a given mesh $\cT$. For simplicity, we consider here the operator $\cL_\cT$ defined such that
\begin{align*}
(\beta\cL_\cT u,w) &= (\alpha\nabla u, \nabla w) &&\forall w\in\cU_\cT,
\end{align*}
where $(\cdot,\cdot)$ denotes the $L^2(\Omega)$ or $L^2(\Omega)^d$ inner product. A slightly different discretisation that allows for an explicit expression of $\cL_\cT$ is given in Section \ref{sec:discretisation}.
Finally, let $\Pi_{j}$ denote a projection operator that projects into the space $\cU_{\cT_j}$. The semi-discrete finite element formulation can be stated as follows: for $j=1,2,\dots$, find $u_j:[T_{j-1},T_j]\rightarrow \cU_{\cT_j}$ such that
\begin{subequations}
\label{eq:FEM}
\begin{align}
\partial_t^2 u_{j} + \cL_{\cT_j}u_{j} &= f_{\cT_j} &&\text{in }\Omega\times (T_{j-1},T_j), \\
u_{j}(\cdot,T_{j-1}) &= \Pi_{j} u_{j-1}(\cdot,T_{j-1}), && \\
\partial_tu_{j}(\cdot,T_{j-1}) &= \Pi_{j}\partial_tu_{j-1}(\cdot,T_{j-1}), &&
\end{align}
\end{subequations}
with $u_{0}(\cdot,T_0)\equiv 0$, $\partial_tu_{0}(\cdot,T_0)\equiv 0$, and $f_{\cT_j}(\cdot,t)\in\cU_{\cT_j}$ a discretisation of $f(\cdot,t)$.

For the time discretisation, we consider here the central difference scheme, although the adaptive method can also be applied to other time integration schemes. Let $\Delta t = T_{up}/m$ denote the time step size, with $m>0$ some positive integer. Also, let $t^n:=t_0+n\Delta t$ and  $u_j^n:=u_j(\cdot,t^n)$. We approximate $\partial_t^2u_j(\cdot,t^n)$ by the central difference scheme
\begin{align}
\label{eq:Dt2}
D_t^2u^n_j := \frac{u^{n+1}_j-2u^n_j+u^{n-1}_j}{\Delta t^2}.
\end{align}
Now, let $n_j:=mj$. The fully discrete finite element formulation can be stated as follows: for $j=1,2,\dots$ and for $n:n_{j-1}-1\leq n \leq n_j$, find $u_j^n\in\cU_{\cT_j}$ such that
\begin{subequations}
\label{eq:FEM2}
\begin{align}
D_t^2 u_{j}^n + \cL_{\cT_j}u_{j}^n &= f_{\cT_j}(\cdot,t^n) &&\text{for }n:n_{j-1}\leq n \leq n_j-1, 
\label{eq:FEM2a} \\
u_{j}^{n} &= \Pi_{j} u_{j-1}^{n} &&\text{for }n=n_{j-1}-1 \text{ and }n=n_{j-1},
\end{align}
\end{subequations}
with $u_0^0\equiv 0$ and $u_0^{-1}\equiv 0$. We can rewrite equation \eqref{eq:FEM2a} as
\begin{align}
\label{eq:FEM2a2}
u^{n+1}_j &= -u^{n-1}_j+2u^n_j + \Delta t^2(-\cL_{\cT_j}u^n_j + f_{\cT_j}(\cdot,t^n)).
\end{align}
An extended time stepping scheme, that takes into account an absorbing boundary layer and the discretisation of $f$, is given in Section \ref{sec:discretisation}.

In practice, we cannot solve the wave equation for $t\rightarrow\infty$, but have to stop at some finite time $t_{stop}=T_{j_{stop}}$. 
To determine $j_{stop}$, we use a stopping criterion of the form
\begin{align}
\label{eq:stop}
\sup_{\vx\in\Omega} |u_j^{{n_j}}(\vx)| \leq \epsilon_0,
\end{align}
where $\epsilon_0>0$ is an \textit{a priori} chosen threshold value. In other words, we stop the computations when the scattered wave field is close to zero, which means it has almost completely left the computational domain.

An overview of how to implement the adaptive finite element method is given in Algorithm~\ref{alg:AFEM}. Here, we use the following functions:
\begin{itemize}
  \item $\cT_j=\Call{updateMesh}{\cT_{j-1},u_{j-1}^n,t^n}$: computes
    the new mesh $\cT_j$ given the current mesh $\cT_{j-1}$, the
    current discrete wave field $u_{j-1}^n$, and the current time
    $t^n$. A detailed description of $\Call{updateMesh}{}$ and how to
    choose the initial mesh $\cT_0$ is given in Section \ref{sec:mesh} below.
  \item $u_j^n:=\Call{project}{u_{j-1}^n,\cT_{j-1},\cT_j}$: computes the projection $u_j^n=\Pi_{j}u^n_{j-1}$. 
  \item $u_j^{n+1} = \Call{doTimeStep}{u_j^n,u_j^{n-1},\cT_j,t^n}$:
    computes the wave field at the next time step 
    $u_j^{n+1}$ using the formula in \eqref{eq:FEM2a2}.
  \item $\Call{Stop}{u_{j}^n,\cT_j,t^n}$: returns $\mathbf{true}$ if
    $t^n>t_f$ and if the stopping criterion given in \eqref{eq:stop} is satisfied. Returns $\mathbf{false}$ otherwise.
\end{itemize} 

\begin{algorithm}
\caption{solving the wave equation using a time-dependent mesh}
\label{alg:AFEM}
\begin{algorithmic}
\Procedure{solveWaveEquation}{}
\State $\cT_h\gets\cT_0$ \Comment{set initial mesh}
\State $u_h \gets 0$, $u_h^{new}\gets 0$, and $u_h^{old} \gets 0$ \Comment{initialise wave field}
\For{$j=1,2,\dots$} 
  \State $n=n_{j-1}$ \Comment{at this point, $u_h=u_{j-1}^{n}$, $u_h^{old}=u_{j-1}^{n-1}$, $\cT_h=\cT_{j-1}$}
  \If{$\Call{stop}{u_h,\cT_h,t^n}$}
    \State \Return{$u_h$}
  \EndIf
  \State $\cT_h^{new} \gets \Call{updateMesh}{\cT_h,u_h,t^n}$ \Comment{$\cT_h^{new}\gets\cT_{j}$}
  \State $u_h \gets \Call{project}{u_h,\cT_h,\cT_h^{new}}$ \Comment{$u_h\gets u_j^n$}
  \State $u_h^{old} \gets \Call{project}{u_h^{old},\cT_h,\cT_h^{new}}$ \Comment{$u_h^{old}\gets u_j^{n-1}$}
  \State $\cT_h \gets \cT_h^{new}$ \Comment{$\cT_h \gets \cT_j$}
  \For{$\ell=0,1,2,\dots,m-1$}
    \State $n\gets n_{j-1}+\ell$ \Comment{at this point, $u_h=u_j^n$, $u_h^{old}=u_j^{n-1}$}
    \State $u_h^{new}\gets \Call{doTimeStep}{u_h,u_h^{old},\cT_h,t^n}$ \Comment{$u_h^{new}\gets u^{n+1}_j$}
    \State $u_h^{old}\gets u_h$ 
    \State $u_h\gets u_h^{new}$ 
  \EndFor
\EndFor 
\EndProcedure
\end{algorithmic}
\end{algorithm}

\subsection{Adapting the mesh}
\label{sec:mesh}
To construct adapted meshes $\cT_j$, we define \textit{a priori} a set of nested 
meshes $\{\cT^1,\cT^2,\dots,\cT^K\}$, $K\geq 2$, where $\cT^1$ is the coarsest mesh with a mesh width $h_1$ independent of $\omega$, and $\cT^K$ is the finest mesh with a mesh width $h_K$ of order $\omega^{-1}$ or less. We assume that, in case of no mesh adaptation, the mesh $\cT^K$ is sufficiently fine for solving the wave equation with the desired accuracy.

We construct adapted meshes $\cT_j$ from elements in $\bigcup_{k=1}^K \cT^k$. 
The initial mesh $\cT_0$ is chosen as the finest mesh $\cT^K$.
The algorithm $\cT_j=\Call{updateMesh}{\cT_{j-1},u_{j-1}^n,t^n}$ for updating the mesh is given in Algorithm \ref{alg:updateMesh}, which consists of the following functions:
\begin{itemize}
  \item $\cP=\Call{getParentElements}{\cT}$: returns the set of all elements in $\bigcup_{k=1}^{K-1}\cT^k$ that are coarser than those of the given mesh $\cT$
  , i.e. it returns
\begin{align*}
\cP &=  \{ \e\in\bigcup_{k=1}^{K-1} \cT^k\setminus\cT \;|\; \e\supset \e' \text{ for some }\e'\in\cT \}.
\end{align*}
  \item $\cT=\Call{getChildElements}{\cP}$: returns the mesh $\cT\subset\bigcup_{k=1}^K \cT^k$, given its parent elements $\cP=\Call{getParentElements}{\cT}$. In particular, 
    $\Call{getChildElements}{}$ returns
$\cT = \bigcup_{\e\in\cP} \Call{getSubElements}{\e}\setminus\cP$.
  \item $\cT_\e=\Call{getSubelements}{\e}$: returns, for a given
    element $\e\in\cT^k$ with $k\leq K-1$, the set of the elements in $\cT^{k+1}$ that are a subset of $\e$. 
  \item $\cP^{*}_{j-1}=\Call{markElements}{\cP_{j-1},u_{j-1}^n,t^n}$: returns the set of all elements in $\cP_{j-1}$ that need to be refined. We mark all elements $\e\in\cP_{j-1}$ that also have subelements in $\cP_{j-1}$. For the elements $\e\in\cP_{j-1}$ that have no further subelements in $\cP_{j-1}$, we only mark those for which the function $\Call{needsRefinement}{\e,u^n_{j-1},t^n}$ returns $\mathbf{true}$. Pseudocode of the function $\Call{markElements}{}$ is given in Algorithm \ref{alg:mark}.
  \item $\Call{needsRefinement}{\e,u_{j-1}^n,t^n}$: returns $\mathbf{true}$ if and only if
\begin{align}
\label{eq:atSupport} 
\e\cap\text{support}(u_I(\cdot,t^n)) \neq \emptyset,
\end{align}
or
\begin{align}
\label{eq:projErr} 
\eta_\e := \sup_{\vx\in\e} |u_{j-1}^{n}(\vx) - \Pi_\e u^n_{j-1}(\vx)| > \eta_0,
\end{align}
where $\Pi_\e$ denotes a projection operator that projects into 
the discrete space of $\e$ and $\eta_0>0$ is some threshold value defined \textit{a priori}.  
  \item $\cP_j=\Call{markNearbyElements}{\cP^{*}_{j-1}}$: for $k=1,2,\dots,K-1$, returns all elements in $\cT^k$ that are within a distance $c_{\max}T_{up}$ of an element in $\cP^{*}_{j-1}\cap\cT^k$.
  The distance between two elements $\e_1$ and $\e_2$ is defined as $\dist(\e_1,\e_2):=\inf_{\vx\in\e_1,\vy\in\e_2} |\vx-\vy|$.
  \end{itemize}

An illustration of the mesh adaptation algorithm is given in Figure \ref{fig:meshUpdate} below.

\begin{algorithm}
\caption{update the mesh}
\label{alg:updateMesh}
\begin{algorithmic}
\Function{updateMesh}{$\cT_{j-1},u^n_{j-1},t^n$}
  \State $\cP_{j-1} \gets \Call{getParentElements}{\cT_{j-1}}$
  \State $\cP^{*}_{j-1}\gets \Call{markElements}{\cP_{j-1},u^n_{j-1},t^n}$
  \State $\cP_j\gets\Call{markNearbyElements}{\cP^{*}_{j-1}}$
  \State $\cT_j \gets \Call{getChildElements}{\cP_j}$
  \State \Return{$\cT_j$}
\EndFunction
\end{algorithmic}
\end{algorithm}

\begin{algorithm}
\caption{mark elements for refinement}
\label{alg:mark}
\begin{algorithmic}
\Function{markElements}{$\cP_{j-1},u^n_{j-1},t^n$}
  \State $\cP^{*}_{j-1} \gets \emptyset$ \Comment{initialise the set of marked elements}
  \For{$\e\in\cP_{j-1}$} 
    \If{$\Call{getSubelements}{\e}\cap{\cP_{j-1}} \neq \emptyset$}
      \State $\cP^{*}_{j-1}\gets\cP^{*}_{j-1}\cup\e$ \Comment{mark $\e$}
    \ElsIf{$\Call{needsRefinement}{\e,u^n_{j-1},t^n}$}
      \State $\cP^{*}_{j-1}\gets\cP^{*}_{j-1}\cup\e$ \Comment{mark $\e$}
    \EndIf  
  \EndFor
  \State \Return{$\cP^{*}_{j-1}$}
\EndFunction
\end{algorithmic}
\end{algorithm}

\subsection{Computing the Fourier Transformation}
\label{sec:FT}
We can approximate the Fourier transform in \eqref{eq:U_S} by a discrete Fourier transformation:
\begin{align*}
U_S=\Ft[u_S](\cdot,-\omega)=\int_{t_0}^\infty e^{\im\omega t}u_S(\cdot,t) \;\dd t \approx  \sum_{n=1}^{\infty} \Delta t e^{\im\omega t^n} u_S(\cdot,t^n).
\end{align*}
Furthermore, we can approximate $u_S(\cdot,t^n)$ by the finite element
approximation $u_h^n$, where $u_h^n:=u_j^n$ for $n:n_{j-1}+1\leq n
\leq n_j$ and where $u_j^n$ is the solution to the fully discrete problem
formulated in \eqref{eq:FEM2}. We assume that $u_S(\cdot,t^n)\approx 0$ in $\Omega$ for $n> n_{stop}:=n_{j_{stop}}$. We then obtain the approximation  
\begin{align*}
U_S \approx U_h^{n_{stop}} := \sum_{n=1}^{n_{stop}} \Delta t e^{\im\omega t^n} u_h^n.
\end{align*} 
We can compute $U_h^{n_{stop}}$ by setting $U_h^0\equiv 0$ in $\Omega$ and by using the recursive relation
\begin{align}
\label{eq:U_h1}
U_h^{n} &= U_h^{n-1} + \Delta t e^{\im\omega t^{n}}u_h^{n} &&\text{for }n=1,2,\dots,n_{stop}.
\end{align}

To accurately approximate $U_S$ on the entire computational domain
$\Omega$, we need to compute $U_h^{n_{stop}}$ on a globally fine mesh
$\cT^K$, which means that, if we use the formula in \eqref{eq:U_h1},
we would need to evaluate $u_h^n$ at each time step at all the degrees
of freedom of $\cT^K$. Since $u_h^n$ is only known at the degrees of
freedom of some adapted mesh $\cT_j$, this means we would need to
interpolate $u_h^n$ at the degrees of freedom of $\cT^K$ at each time
step. We can significantly reduce the average
number of interpolation points per time step in the computation of
$U_h^{n_{stop}}$ by using an adapted space-time mesh. 

Let $\cQ$ be a space-time mesh for the space-time domain $\Omega\times(T_0,T_{j_{stop}})$ with space-time elements $Q$ of the form $Q=\e_Q\times(T_{j_{Q,0}},T_{j_{Q,1}})$. We choose the space-time elements such that $\e_Q\in\cT_{j}$ for $j:j_{Q,0}<j\leq j_{Q,1}$. Let $\chi_{Q}(\vx,t)$ be the characteristic function given by
\begin{align*}
\chi_{Q}(\vx,t) := \begin{cases}
\chi_{\e_Q}, & t\in(T_{j_{Q,0}},T_{j_{Q,1}}], \\
0, &\text{otherwise},
\end{cases}
\qquad
\chi_\e(\vx) := \begin{cases}
1, &\vx\in \e\\
0, &\text{otherwise}.
\end{cases}
\end{align*}
A discrete version of $\chi_\e$ is given in Section
\ref{sec:discretisation} below. We have the partition of
unity property
\begin{align}
\label{eq:PU}
\sum_{Q\in\cQ} \chi_Q(\vx,t^n) &= 1 &&\text{for a.e. }\vx\in\Omega, \forall n:1\leq n\leq n_{stop}.
\end{align}
We can therefore write
\begin{align}
U_h^{n_{stop}} &= \sum_{n=1}^{n_{stop}} \Delta t e^{\im\omega t^n} u_h^n  
\overset{\eqref{eq:PU}}{=} \sum_{n=1}^{n_{stop}} \left(\sum_{Q\in\cQ} \chi_Q(\cdot,t^n)\right) \Delta t e^{\im\omega t^n} u_h^n  \nonumber \\
&= \sum_{Q\in\cQ} \left(\sum_{n=1}^{n_{stop}} \chi_Q(\cdot,t^n) \Delta t e^{\im\omega t^n} u_h^n\right) 
=: \sum_{Q\in\cQ} \Delta U_Q. \label{eq:U_h2}
\end{align}
Note that $\Delta U_Q$ has support only in $\e_Q$. 
With a slight abuse of notation, we
let $\Delta U_Q$ also denote its restriction to $\e_Q$.
Which definition is used will be clear from the context.

Let $n_{Q,0}:=n_{j_{Q,0}}$ and $n_{Q,1}:=n_{j_{Q,1}}$. We can compute
$\Delta U_Q$ by first setting $\Delta U_Q=\Delta U_Q^{n_{Q,1}}$, where
$\Delta U_Q^{n_{Q,1}}$
is computed by setting $\Delta U_Q^{n_{Q,0}}= 0$ in $\e_Q$ and by using the recursive relation
\begin{align*}
\Delta U_{Q}^{n} &= \Delta U_Q^{n-1} + \Delta te^{\im\omega t^{n}} u_h^{n}|_{E_Q} &&\text{for }n:n_{Q,0}+1\leq n \leq n_{Q,1}.
\end{align*}
Note that, in order to compute $\Delta U_Q$, we only need the values
of $u_h^n$ at the degrees of freedom corresponding to the spatial element $\e_Q$. To compute $U_h^{n_{stop}}= \sum_{Q\in\cQ} \Delta U_Q$, we need to interpolate $\Delta U_Q$ at the degrees of freedom of the finest mesh $\cT^K$ at ${E_Q}$. However, we only need to do this once for each space time element.

To minimise the computational cost, we choose the space-time elements as large as possible. This means that we choose the time intervals $(T_{Q,0},T_{Q,1})$ as large as possible, namely such that $\e_Q\in\cT_j$ for all $j:j_{Q,0}<j\leq j_{Q,1}$ and such that $\e_Q\notin \cT_j$ for $j=j_{Q,0}$ and $j=j_{Q,1}+1$. An illustration of a space-time mesh is given in Figure~\ref{fig:spaceTimeMesh} below.

In practice, we do not need to construct the space-time mesh explicitly. Let $Q(j,\e)$ denote the unique space-time element $Q\in\cQ$ such that $\e=\e_Q$ and $j_{Q,0}<j\leq j_{Q,1}$. We define $\Delta U_{j,\e}:=\Delta U_{Q(j,\e)}$ and $\Delta U_{j,\e}^n:=\Delta U_{Q(j,\e)}^n$. 
We also define 
\begin{align*}
U_h^{n_j} := \sum_{Q\in\cQ:\; j_{Q,1} \leq j} \Delta U_Q.
\end{align*}
We have that $U_h^0 \equiv 0$ and we can compute $U_h^{n_{stop}}$ by
computing,
for $j=1,2,\dots,j_{stop}$,
the following:
\begin{align*}
U_h^{n_j}&= U_h^{n_{j-1}} + \sum_{Q\in\cQ: \;j_{Q,1}=j} \Delta U_{Q},
\end{align*}
or, equivalently,
\begin{align}
\label{eq:U_h3}
U_h^{n_j}&=  U_h^{n_{j-1}} + \sum_{\e\in\cT_j\setminus\cT_{j+1}} \Delta U_{j,\e},
\end{align}
with $\cT_{j_{stop}+1}:=\emptyset$. We can compute $\Delta U_{j,\e} = \Delta U_{j,\e}^{n_j}$ by first setting $\Delta U_{0,\e}^0\equiv 0$ for all $\e\in\cT_0$ and by computing, for $j=1,2,\dots,j_{stop}$, the following:
\begin{subequations}
\label{eq:DU}
\begin{align}
\Delta U_{j,\e}^{n} &= \Delta U_{j,\e}^{n-1} + \Delta t e^{\im \omega t^{n}} u_h^{n}|_{\e} &&\forall \e\in\cT_j \And n:n_{j-1}+1\leq n\leq n_{j}, \\
\Delta U_{j,\e}^{n_{j-1}} &\equiv 0 &&\forall\e\in \cT_j\setminus\cT_{j-1}, \\
\Delta U_{j,\e}^{n_{j-1}} &= \Delta U_{j-1,\e}^{n_{j-1}} && \forall\e\in \cT_j\cap\cT_{j-1}.
\end{align}
\end{subequations}

Therefore, letting 
$\Delta U_j^n:=\{\Delta U_{j,\e}^n\}_{\e\in\cT_j}$, the algorithm for computing $U_h^{n_{stop}}$ is given in Algorithm \ref{alg:FT}, which consists of the following functions:
\begin{itemize}
  \item $U_h=\Call{updateFT}{U_h,\Delta U_{j-1}^n,\cT_{j-1},\cT_j}$: computes $U_h|_\e\gets U_h|_\e + \Delta U_{j-1,\e}^n$ for all $\e\in\cT_{j-1}\setminus\cT_j$.
  \item $\Delta U_j^n=\Call{initialiseNewIncrements}{\Delta U_{j-1}^n, \cT_{j-1},\cT_j}$: sets $\Delta U_{j,\e}^n\gets 0$ for all $\e\in\cT_j\setminus\cT_{j-1}$ and $\Delta U_{j,\e}^n\gets \Delta U_{j-1,\e}^{n}$ for all $\e\in \cT_j\cap\cT_{j-1}$.
  \item $\Delta U_j^n = \Call{updateIncrements}{\Delta U_{j}^{n-1}, u_h^n, \cT_j,t^n}$: computes $U_{j,\e}^n=U_{j,\e}^{n-1}+\Delta t e^{\im\omega t^n}u_h^n|_\e$ for all $\e\in\cT_j$.
\end{itemize}

\begin{algorithm}
\caption{compute the discrete Fourier transform with respect to time}
\label{alg:FT}
\begin{algorithmic}
\Function{computeFT}{$\{u_h^n\}_{n=1}^{n_{stop}},\{\cT_j\}_{j=1}^{j_{stop}}$}
  \State $U_h\gets 0$ \Comment{initialise Fourier transform}
  \State $\Delta U_h\gets 0$ \Comment{initialise increments}
  \For{$j=1,2,\dots,j_{stop}$}
    \State $n\gets n_{j-1}$ \Comment{at this point, $U_h=U_h^{n_{j-2}}$, $\Delta U_h=\Delta U_{j-1}^{n}$}
    \State $U_h \gets \Call{updateFT}{U_h,\Delta U_h,\cT_{j-1},\cT_{j}}$ \Comment{$U_h\gets U_h^{n}$}
    \State $\Delta U_h \gets \Call{initialiseNewIncrements}{\Delta U_h,\cT_{j-1},\cT_j}$ \Comment{$\Delta U_h\gets \Delta U_j^{n}$}
    \For{$\ell=1,2,\dots,m$}
      \State $n\gets n_{j-1}+\ell$
      \State $\Delta U_h \gets \Call{updateIncrements}{\Delta U_h,u_h^n,\cT_j,t^n}$ \Comment{$\Delta U_h\gets \Delta U_j^{n}$}
    \EndFor
  \EndFor
  \State $U_h \gets \Call{updateFT}{U_h,\Delta U_h, \cT_{j_{stop}},\emptyset}$ \Comment{$U_h\gets U_h^{n_{stop}}$}
  \State \Return{$U_h$}
\EndFunction
\end{algorithmic}
\end{algorithm}

\subsection{Overview of the complete algorithm}
\label{sec:overview}
To approximate the solution $U_S$ to the Helm\-holtz equation given in
\eqref{eq:Helmholtz2}, we approximate the solution $u_S$ to the
time-dependent wave equation in \eqref{eq:wave_S2} using Algorithm
\ref{alg:AFEM} and then approximate the Fourier transform
$\Ft[u_S](\cdot,-\omega)$ using Algorithm \ref{alg:FT}. We can solve
the wave equation and compute the Fourier transform simultaneously,
resulting in Algorithm \ref{alg:Helmholtz}. This last algorithm gives
a complete
overview of the proposed method 
for solving the Helmholtz equation.

\begin{algorithm}
\caption{solving the Helmholtz equation using a time-domain approach}
\label{alg:Helmholtz}
\begin{algorithmic}
\Procedure{solveHelmholtzEquation}{}
\State $\cT_h\gets\cT^K$ \Comment{set initial mesh}
\State $u_h \gets 0$, $u_h^{new}\gets 0$, and $u_h^{old} \gets 0$ \Comment{initialise wave field}
\State $U_h\gets 0$ \Comment{initialise Fourier transform}
\State $\Delta U_h\gets 0$ \Comment{initialise increments}
\For{$j=1,2,\dots$} 
  \State $n=n_{j-1}$ \Comment{at this point, $\cT_h=\cT_{j-1}$, $u_h=u_{j-1}^{n}, u_h^{old}=u_{j-1}^{n-1}$}
  \State \Comment{also, at this point, $U_h=U_h^{n_{j-2}}$, $\Delta U_h=\Delta U_{j-1}^n$}
  \If{$\Call{stop}{u_h,\cT_h}$}
    \State $U_h \gets \Call{updateFT}{U_h,\Delta U_h, \cT_h,\emptyset}$ \Comment{$U_h\gets U_h^{n}$}
    \State \Return{$U_h$}
  \EndIf
  \State $\cT_h^{new} \gets \Call{updateMesh}{\cT_h,u_h,t^n}$ \Comment{$\cT_h^{new}\gets\cT_{j}$}
  \State $U_h \gets \Call{updateFT}{U_h,\Delta U_h,\cT_{h},\cT_{h}^{new}}$ \Comment{$U_h\gets U_h^{n}$}
  \State $\Delta U_h \gets \Call{initialiseNewIncrements}{\Delta U_h,\cT_h,\cT_h^{new}}$ \Comment{$\Delta U_h\gets \Delta U_j^{n}$} 
  \State $u_h \gets \Call{project}{u_h,\cT_h,\cT_h^{new}}$ \Comment{$u_h\gets u_j^n$}
  \State $u_h^{old} \gets \Call{project}{u_h^{old},\cT_h,\cT_h^{new}}$ \Comment{$u_h^{new}\gets u_j^{n-1}$}
  \State $\cT_h \gets \cT_h^{new}$ \Comment{$\cT_h \gets \cT_j$}
  \For{$\ell=0,1,2,\dots,m-1$}
    \State $n\gets n_{j-1}+\ell$ \Comment{at this point, $u_h=u_j^n, u_h^{old}=u_j^{n-1}, \Delta U_h = \Delta U_j^n$}
    \State $u_h^{new}\gets \Call{doTimeStep}{u_h,u_h^{old},\cT_h,t^n}$ \Comment{$u_h^{new}\gets u^{n+1}_j$}
    \State $u_h^{old}\gets u_h$ \Comment{$u_h^{old}\gets u^{n}_j$}
    \State $u_h\gets u_h^{new}$ \Comment{$u_h\gets u^{n+1}_j$}
    \State $\Delta U_h \gets \Call{updateIncrements}{\Delta U_h,u_h,\cT_h,t^{n+1}}$ \Comment{$\Delta U_h\gets \Delta U_j^{n+1}$}
  \EndFor
\EndFor 
\EndProcedure
\end{algorithmic}
\end{algorithm}

\section{Numerical Examples}
\label{sec:numerics}
We present numerical examples for wave scattering problems in 1 and 2 spatial dimensions. Details of the finite element discretisation are provided for the 2-dimensional case. The discretisation for the 1-dimensional case can be readily deduced from the 2-dimensional case. All the numerical experiments presented in this section have been carried out in MATLAB R2017a.

\subsection{Absorbing boundary layer and finite element discretisation}
We start by specifying how we impose 
an absorbing boundary layer
and then provide details of the finite element discretisation.

\subsubsection{Absorbing boundary layer}
We consider the wave equation in \eqref{eq:wave_S3} for a rectangular
domain and add an additional absorbing boundary layer. Let $\Omega_0 =
(-L_1,L_1)\times(-L_2,L_2)\supset\Omega_{in}$ be the region of
interest, and let $\Omega_{ABL}$ be an additional absorbing boundary layer
of width $W$ surrounding $\Omega_0$. We define 
the computational domain by
$\Omega={\Omega}_0\cup\Omega_{ABL}=(-L_1-W,L_1+W)\times(-L_2-W,L_2+W)$. We
apply the perfectly matched layer 
that was introduced in \cite{grote2010} and further analysed in \cite{kaltenbacher2013,baffet2019}. The resulting wave equation is given by
\begin{subequations}
\label{eq:PML}
\begin{align}
\partial_t^2 u + (\zeta_1+\zeta_2)\partial_tu + \zeta_1\zeta_2u- \beta^{-1}\nabla\cdot(\alpha(\nabla u + \vs)) &= f &&\text{in }\Omega\times(t_0,\infty), \\
\partial_t \vs + Z_1\vs + Z_2\nabla u &= \vzero &&\text{in }\Omega\times(t_0,\infty), \\ 
u(\cdot,t_0)=\partial_tu(\cdot,t_0)&= 0 &&\text{in }\Omega, \\
\vs(\cdot,t_0) &= \vzero &&\text{in }\Omega, \\
u &= 0 &&\text{on }\partial\Omega\times(t_0,\infty),
\end{align}
\end{subequations}
with
\begin{align*}
Z_1 := \begin{bmatrix} \zeta_1 & 0 \\ 0 & \zeta_2 \end{bmatrix} \text{ and }
Z_2 := \begin{bmatrix} \zeta_1-\zeta_2 & 0 \\ 0 & \zeta_2-\zeta_1 \end{bmatrix}. 
\end{align*}
Here, $\vs=\vs(\vx)=(s_1(\vx),s_2(\vx))$ is an auxiliary vector field
that has support only in $\Omega_{ABL}$, and $\zeta_1=\zeta_1(x_1)\geq 0$ and $\zeta_2=\zeta_2(x_2)\geq 0$ are two additional parameters that are nonzero only in $\Omega_{ABL}$. In particular, $\zeta_1(x_1)$ is nonzero only for $x_1:L_1\leq|x_1|\leq L_1+W$ and $\zeta_2(x_2)$ is nonzero only for $x_2:L_2\leq|x_2|\leq L_2+W$.

\subsubsection{Finite element discretisation}
\label{sec:discretisation} 
To construct the adapted meshes, we use a set of nested rectangular meshes $\{\cT^k\}_{k=1}^K$ that are constructed as follows. First, we define a sequence of mesh widths $h_1>h_2>\dots>h_K$. Mesh $\cT^k$ is then constructed using elements of size $h_{k,1}\times h_{k,2}$, where $h_{k,i}=h_k$ in the region $|x_i|<L_i$ and $h_{k,i}=h_K$ in the region $|x_i|>L_i$, for $i=1,2$. In other words, $\cT^k$ is a Cartesian mesh of width $h_k$ in the main domain $\Omega_0$, but has thin elements of the finest resolution $h_K$ in the absorbing boundary layer $\Omega_{ABL}$. An illustration of the nested meshes is given in Figure \ref{fig:meshes}.

\begin{figure}[h]
\centering
\includegraphics[width=0.3\textwidth]{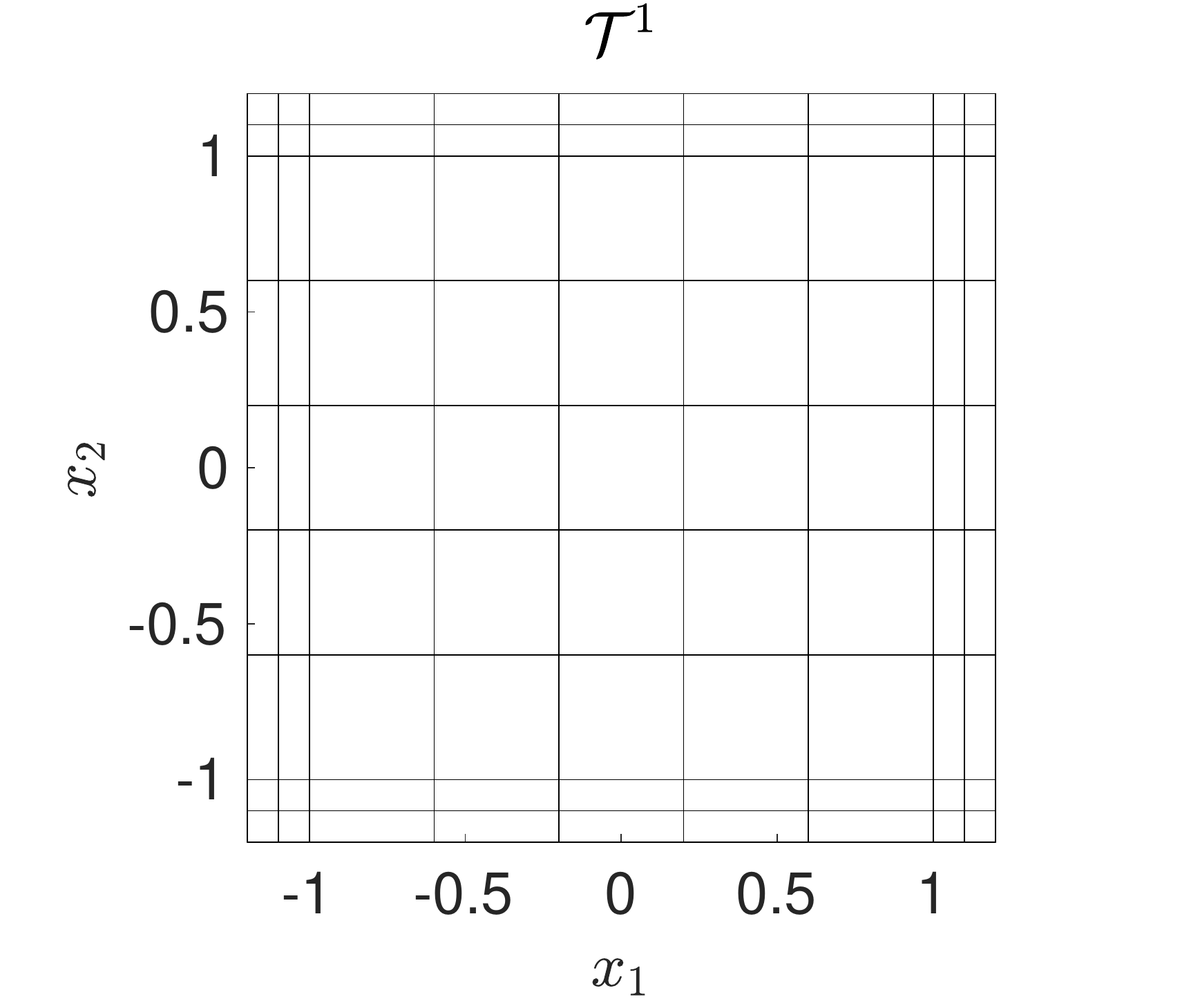}\,
\includegraphics[width=0.3\textwidth]{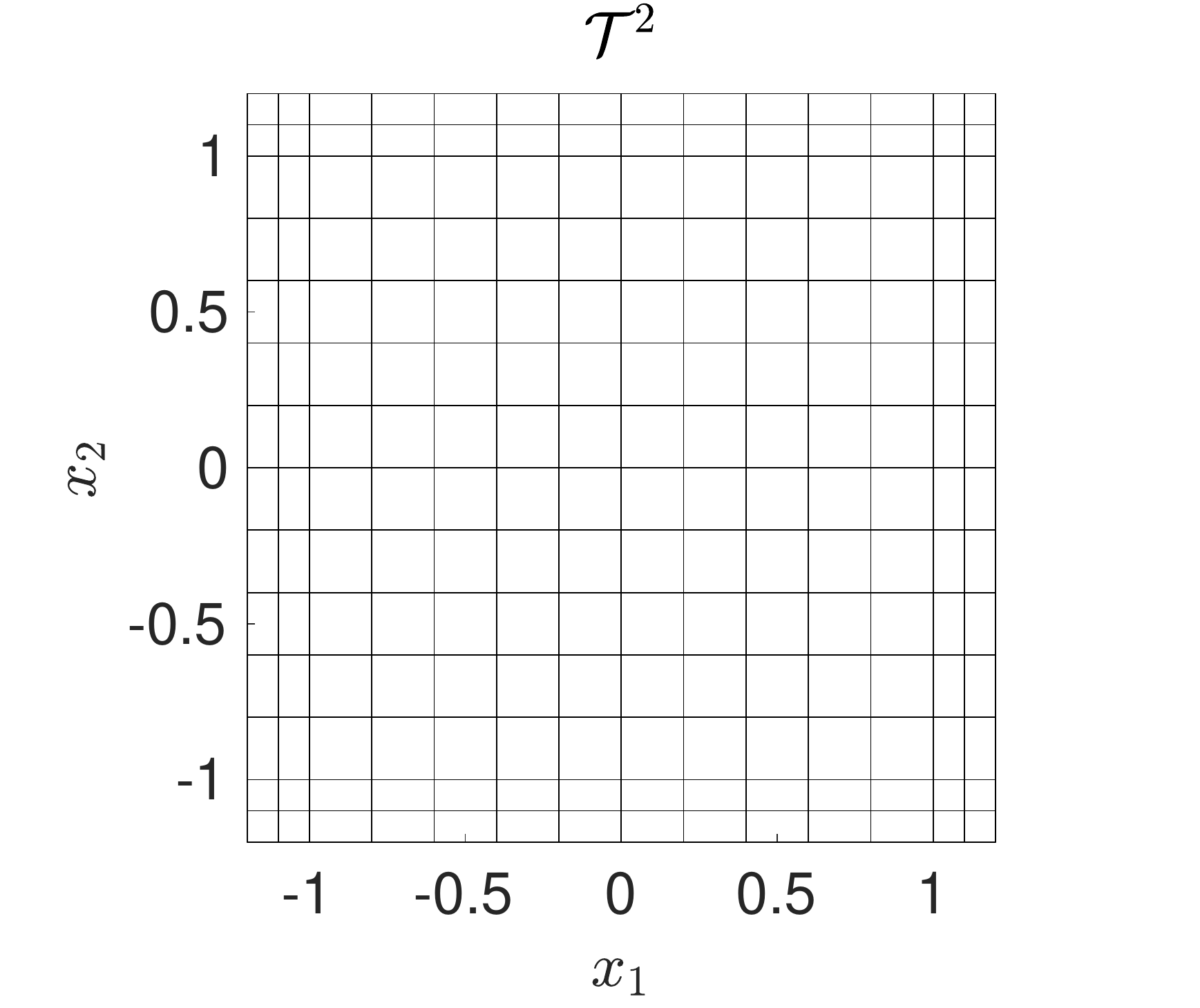}\,
\includegraphics[width=0.3\textwidth]{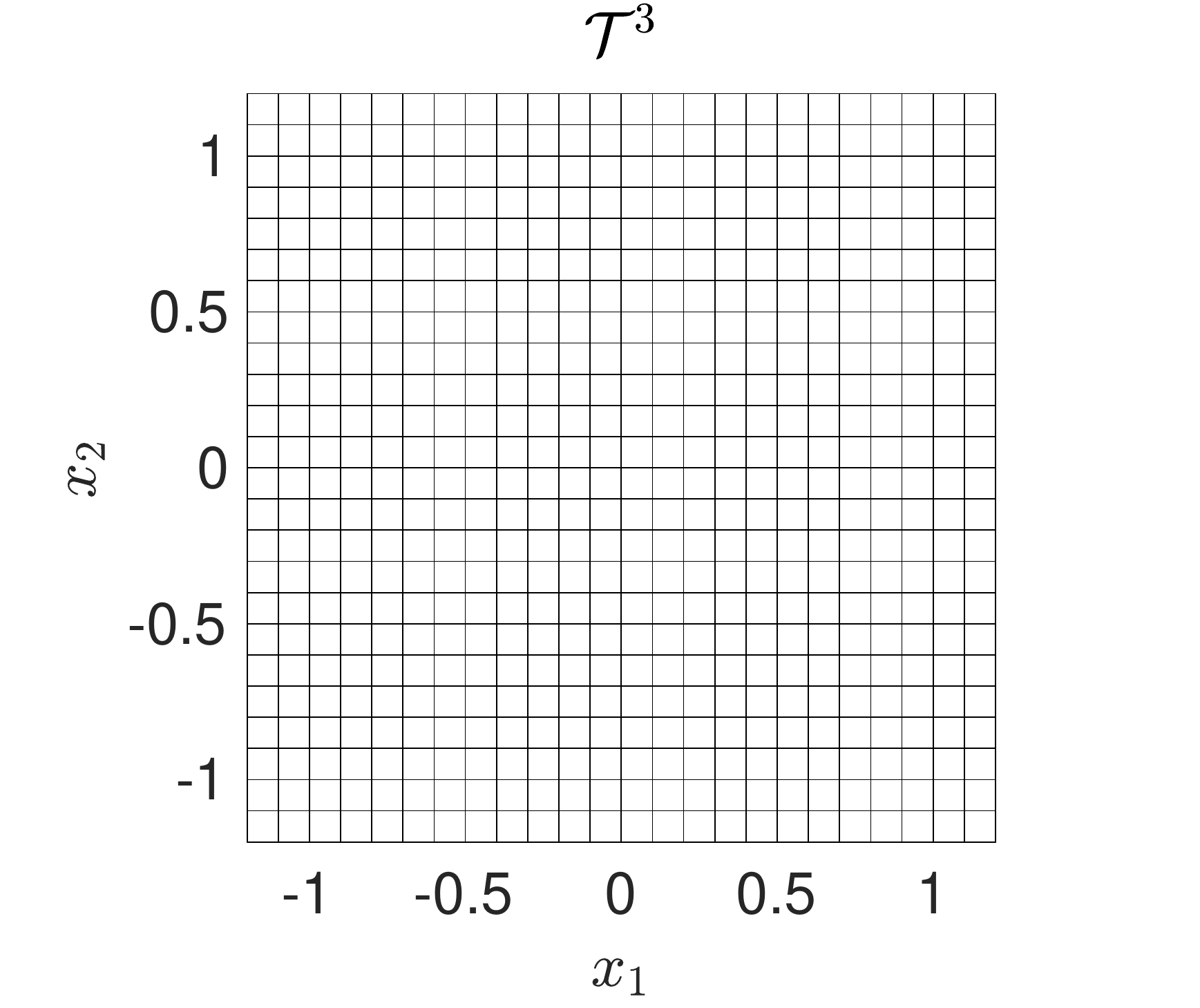}
\caption{Illustration of $\cT^k$, $k=1,2,3$, with $\{h_k\}=\{\frac25,\frac15,\frac1{10}\}$, $\Omega_0=(-1,1)^2$, $W=\frac15$.}
\label{fig:meshes}
\end{figure}

To construct the conforming finite element space $\cU_\cT$ for a given mesh $\cT$, 
we use tensor-product polynomials of degree at most $p$, namely the polynomial 
reference space is $\chU = \text{span}\{ x_1^kx_2^\ell \;|\; k,\ell
\leq p\}$. As degrees of freedom, we use the values at the nodes
$\cX_{\cT}$, where $\cX_{\cT}$ 
consists of the $(p+1)\times(p+1)$ tensor-product Gauss--Lobatto points of each element $\e\in\cT$.
In case $\cT$ has a hanging node, $\cX_{\cT}$ contains the Gauss--Lobatto points corresponding to the coarsest edge adjacent to the hanging node, but not the points corresponding to the finer edges adjacent to the hanging node.

For the discretisation of $\vs$, we use the discontinuous finite element space $\cS_\cT^2$, where $\cS_{\cT}$ is given by
\begin{align*}
\cS_\cT &:= \{ s\in L^2(\Omega) \;|\; s|_{\Omega_0}\equiv0 \text{ and } s\circ\phi_E\in \chU \text{ for all }\e\in\cT:\e\subset\Omega_{ABL}\}.
\end{align*} 
Let $\vs_\e:=\vs|_\e$. The degrees of freedom of $\vs\in\cS_\cT^2$ are given by $\vs_\e(\vx)$ for all $\e\in\cT$ and all $\vx\in\cX_\e$, where $\cX_\e:=\cX_{\{\e\}}$ denotes the set of nodes on $\e$. 

We define the discretisation
$\cL_\cT(u,\vs):(\cU_\cT,\cS_\cT^2)\rightarrow\cU_\cT$ of the spatial
operator $(u,\vs)\mapsto 
-\beta^{-1}\nabla\cdot(\alpha(\nabla u+\vs))$ such that
\begin{align*}
(\beta \cL_\cT(u,\vs),w)_{\cT,ML} &= (\alpha(\nabla u + \vs),\nabla w)_{\cT} &&\forall w\in\cU_\cT.
\end{align*} 
Here, $(\cdot,\cdot)_{\cT}$ denotes the approximation of the $L^2$ inner product $(\cdot,\cdot)$ using the tensor-product $(p+1)$-point Gauss--Lobatto quadrature rule for each element in~$\cT$. Furthermore, $(\cdot,\cdot)_{\cT,ML}$ denotes the approximation of $(\cdot,\cdot)$ using a mass-lumping technique, i.e.
\begin{align*}
(u,w)_{\cT,ML} &= \sum_{\vx\in\cX_{\cT}} u(\vx)w(\vx) \sigma_{\vx,\cT},
\end{align*}
where $\sigma_{\vx,\cT}:=\int_{\Omega} w_{\vx,\cT}(\vy) \,\dd\vy$ and $w_{\vx,\cT}\in\cU_\cT$ denotes the nodal basis function corresponding to node $\vx$. We can give an explicit expression for $\cL_\cT$:
\begin{align*}
\cL_\cT(u,\vs)(\vx) &= \frac{\big(\alpha(\nabla u+\vs),\nabla w_{\vx,\cT}\big)_{\cT}}{\beta(\vx) \sigma_{\vx,\cT}} &&\forall \vx\in\cX_\cT.
\end{align*}

For the projection operators, let $\cT_{j-1}+\cT_j$ denote the mesh constructed from elements of $\cT_{j-1}$ and $\cT_j$ by always selecting the finest elements. We define the projection operators $\Pi_j:\cU_{\cT_{j-1}}\rightarrow\cU_{\cT_j}$ and $\Pi_j^{S}:\cS_{\cT_{j-1}}\rightarrow\cS_{\cT_j}$ in such a way that
\begin{align*}
(\Pi_ju,w)_{\cT_j,ML} &= (u,w)_{\cT_{j-1}+\cT_j} &&\forall w\in\cU_{\cT_j}, \\
(\Pi^S_j s,w)_{\cT_j} &= (s,w)_{\cT_{j-1}+\cT_j} &&\forall w\in\cS_{\cT_j}.
\end{align*}
Furthermore, let $\e\in\cT^k$ with $k\leq K-1$, and let
$\cT_\e=\Call{getSubelements}{\e}$. We define the projection operator
$\Pi_\e:\cU_\e:=\cU_{\{\e\}}\rightarrow\cU_{\cT_\e}$, used for the
refinement criterion in \eqref{eq:projErr}, in such a way that
\begin{align*}
(\Pi_\e u,w)_{\{\e\}} &= (u,w)_{\cT_\e}&&\forall w\in\cU_\e.
\end{align*}
We can give explicit expressions for these projection operators:
\begin{align*}
\Pi_ju(\vx) &= \frac{\big(u,w_{\vx,\cT_j}\big)_{\cT_{j-1}+\cT_j}}{\sigma_{\vx,\cT_j}} &&\forall \vx\in\cX_{\cT_j}, \\
(\Pi_j^Ss)_\e(\vx) &= \frac{\big(s,w_{\vx,\e}\big)_{\cT_{j-1}+\cT_j}}{\sigma_{\vx,\e}} &&\forall \e\in\cT_j \text{ and } \vx\in\cX_{\e}, \\
\Pi_\e u(\vx) &= \frac{(u,w_{\vx,\e})_{\cT_\e}}{\sigma_{\vx,\e}} &&\forall \vx\in\cX_\e.
\end{align*}
Here $\sigma_{\vx,\e}:=\int_\e w_{\vx,\e}(\vy)\,\dd\vy$ and $w_{\vx,\e}\in\cS_\e$ denotes the discontinuous nodal basis function corresponding to element $\e$ and node $\vx$.

For the time discretisation, let $\vs^n:=\vs(\cdot,t^n)$. We define $u^{n+1/2}:=\frac12(u^n+u^{n+1})$ and $\vs^{n+1/2}:=\frac12(\vs^n+\vs^{n+1})$. We also define the discrete second-order time derivative $D_t^2u_j^n$ as in \eqref{eq:Dt2} and we define the discrete time derivatives $D_{2t}u^n$ and $D_t\vs^{n+1/2}$ as 
\begin{align*}
D_{2t}u^n &:= \frac{u^{n+1}-u^{n-1}}{2\Delta t}, \quad
D_{t}\vs^{n+1/2}:= \frac{\vs^{n+1}-\vs^{n}}{\Delta t}.
\end{align*}

The fully discrete finite element formulation can then be stated as
follows: for $j=1,2,\dots$, find $u_j^n\in\cU_{\cT_j}$ for
$n:n_{j-1}-1\leq n\leq n_j$ and $\vs_j^n\in\cS_{\cT_j}$ for
$n:n_{j-1}\leq n \leq n_j$ such that
\begin{subequations}
\label{eq:FEM3}
\begin{align}
D_t^2 u_{j}^n + (\zeta_1+\zeta_2)D_{2t}u_j^n + \zeta_1\zeta_2u_j^n + \cL_{\cT_j}(u_{j}^n,\vs_j^n) &= f_{\cT_j}(\cdot,t^n)  && \text{at all }\vx\in\cX_{\cT_j}, \label{eq:FEM3a} \\
 D_{t} \vs_{j,\e}^{n+1/2} + Z_1\vs_{j,\e}^{n+1/2} + Z_2\nabla u_j^{n+1/2}|_\e &= 0 && \begin{matrix*}[l] \text{at all }\vx\in\cX_{\e}, \\ \forall \e\in\cT_j, \end{matrix*}  \label{eq:FEM3b} 
\end{align}
\end{subequations}
for $n:n_{j-1}+1\leq n \leq n_j$,
and 
\begin{align*}
u_{j}^{n} &= \Pi_{j} u_{j-1}^{n} &&\text{for }n=n_{j-1}-1 \text{ and }n=n_{j-1}, \\
\vs_{j}^{n} &= \Pi^S_{j} \vs_{j-1}^{n} &&\text{for }n=n_{j-1},
\end{align*}
with $u_0^0\equiv 0$, $u_0^{-1}\equiv 0$, and $\vs^0\equiv\vzero$.

We can rewrite \eqref{eq:FEM3} as
\begin{subequations}
\label{eq:FEM32}
\begin{align}
u^{n+1}_j &= \frac{\tz_1u_j^{n-1} + \tz_2u_j^n + \Delta t^2(-\cL_{\cT_j}(u^n_j,\vs^n_j) + f_{\cT_j}(\cdot,t^n)) }{\tz_3} &&\text{at all }\vx\in\cX_{\cT_j}   \\
\vs^{n+1}_{j,\e} &= \tZ_3^{-1} \big( \tZ_1\vs^n_{j,\e} + \tZ_2\nabla u_j^{n+1/2}|_\e \big) && \begin{matrix*}[l] \text{at all }\vx\in\cX_{\e}, \\ \forall \e\in\cT_j, \end{matrix*} 
\end{align}
\end{subequations}
where
\begin{align*}
\tz_1&:= -1+\frac12\Delta t(\zeta_1+\zeta_2),  & \tZ_1 &:=  I-\frac12\Delta tZ_1, \\
\tz_2&:= 2-\Delta t^2\zeta_1\zeta_2,  & \tZ_2 &:=  -\Delta tZ_2, \\
\tz_3&:= 1+\frac12\Delta t(\zeta_1+\zeta_2),& \tZ_3 &:=  I+\frac12\Delta tZ_1,
\end{align*}
with $I\in\mR^{2\times 2}$ the identity matrix. 

For the discretisation of the source term $f$, recall that $f=-\partial_t^2 u_I + \beta^{-1}\nabla\cdot(\alpha\nabla u_I)$. One can check that $-\partial_t^2u_I + \beta_0^{-1}\nabla\cdot(\alpha_0\nabla u_I)\equiv 0$ and therefore we have $-\beta_0\beta^{-1}\partial_t^2u_I+\beta^{-1}\nabla\cdot(\alpha_0\nabla u_I)\equiv 0$. We can therefore write $f=-(\beta-\beta_0)\beta^{-1}\partial_t^2u_I + \beta^{-1}\nabla\cdot((\alpha-\alpha_0)\nabla u_I)$. We can discretise the time- and spatial derivatives in a similar way as before. The discrete source term $f_\cT(\cdot,t^n)\in\cU_\cT$ can then be given by
\begin{align*}
f_{\cT}(\vx,t^n) &:= -\frac{\beta(\vx)-\beta_0}{\beta(\vx)}D_t^2u_I^n(\vx) - \frac{\big((\alpha-\alpha_0)\nabla u_I^n,\nabla w_{\vx,\cT}\big)_{\cT}}{\beta(\vx) \sigma_{\vx,\cT}}  &&\forall\vx\in\cX_{\cT},
\end{align*}
where $u^n_I:=u_I(\cdot,t^n)$. Note that, due to this discretisation, the discrete source term is still zero in the exterior domain $\Omega_{ex}$ and for all $t^n>t_f$.

To discretise the characteristic function $\chi_\e$ for a square element $\e=(x_{1,\e},x_{1,\e}+h_\e)\times(x_{2,\e},x_{2,\e}+h_\e)$, we define
\begin{align}
\label{eq:chiE}
\chi_\e(\vx)=\begin{cases}
1, & \begin{matrix*}[l] x_1\in[x_{1,\e},x_{1,\e}+h_\e)\cup (\{x_{1,\e}+h_\e\}\cap\{L_1+W\}) \text{ and } \\ x_2 \in[x_{2,\e},x_{2,\e}+h_\e)\cup (\{x_{2,\e}+h_\e\}\cap\{L_2+W\}),\end{matrix*} \\
0, & \text{otherwise}.
\end{cases}
\end{align}
Then the partition of unity property \eqref{eq:PU} is valid
for every $\vx\in\Omega$ (not just a.e. $\vx\in\Omega$) and, in particular, for all nodes $\vx\in\cX_{\cT^K}$.

\subsubsection{Algorithm for solving the Helmholtz equation}
To solve the Helmholtz equation given in \eqref{eq:Helmholtz2}, we
use Algorithm \ref{alg:Helmholtz} with a few small modifications given below.
These modifications take the absorbing boundary layer into account and ensure that all the steps are fully computable.

\begin{itemize}
  \item At the start, we also initialise the auxiliary variables $\vs_h\gets \vzero$.
  \item After we compute
    $u_h\gets\Call{project}{u_h,\cT_h,\cT_h^{new}}$, we also compute
    the auxiliary variable $\vs_h\gets\Call{projectS}{\vs_h,\cT_h,\cT_h^{new}}$, where the function $\vs_j^n=\Call{projectS}{\vs_{j-1}^n,\cT_{j-1},\cT_j}$ computes $\vs_j^n=\Pi_j^S\vs_{j-1}^n$.
  \item Instead of computing $u_h^{new}\gets
    \Call{doTimeStep}{u_h,u_h^{old},\cT_h,t^n}$, we now compute the
    pair
    $(u_h^{new},\vs_h)\gets\Call{doTimeStep}{u_h,u_h^{old},\vs_h,\cT_h,t^n}$, where we use the modified function
    $(u_j^{n+1},\vs_j^{n+1})=\Call{doTimeStep}{u_j^n,u_j^{n-1},\vs_j^n,
      \cT_j,t^n}$ that computes $u_j^{n+1}$ and $\vs_j^{n+1}$ with the formulae in \eqref{eq:FEM32}.
  \item For the stopping criterion given in \eqref{eq:stop}, we do not take the supremum over all $\vx\in\Omega$, but instead we compute the supremum over all nodes $\vx\in\cX_{\cT_j}$. Similarly, for the refinement criterion given in \eqref{eq:projErr}, we do not take the supremum over all $\vx\in\e$, but only over all nodes $\vx\in\cX_{\cT_\e}$, where $\cT_\e= \Call{getSubelements}{\e}$.
  \item For the function $U_h=\Call{updateFT}{U_h,\Delta U_{j-1}^n,\cT_{j-1},\cT_j}$, we now compute $U_h(\vx)\gets U_h(\vx)+\chi_\e(\vx)\Delta U_{j-1,\e}^n(\vx)$ for all $\vx\in\cX_{\cT^K}\cap\overline\e$ and all $\e\in\cT_{j-1}\setminus\cT_j$, with $\chi_\e$ defined as in \eqref{eq:chiE}.
\end{itemize}

\subsection{Numerical example in 1D}
\label{sec:num1D}
As a first numerical example, we consider the 1D domain $\Omega=(-1,1)$, with spatial parameters
\begin{align*}
\alpha(x) &= \begin{cases}
1 + 3(1-2x)^2(1+2x)^2, & x\in(-\frac12,\frac12), \\
1, & \text{otherwise},
\end{cases} \\
 \beta(x) &= 1,
\end{align*}
and an incoming plane wave $U_I(x) = e^{\im \omega x/c_0}$, with $c_0=1$. An illustration of $\alpha$ is given in Figure \ref{fig:alpha1D}.

\begin{figure}[h]
\centering
\includegraphics[width=0.45\textwidth]{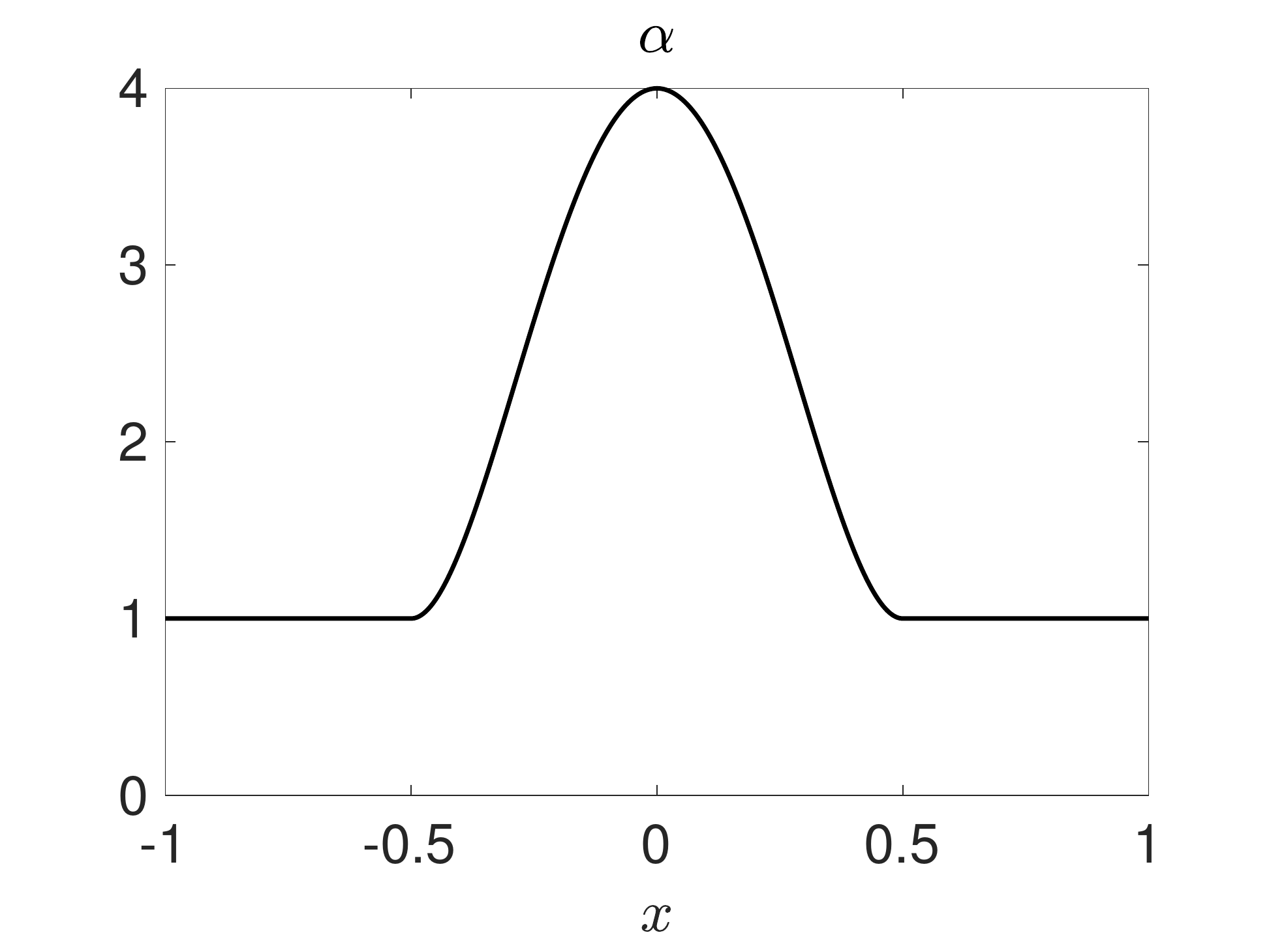}\,
\includegraphics[width=0.45\textwidth]{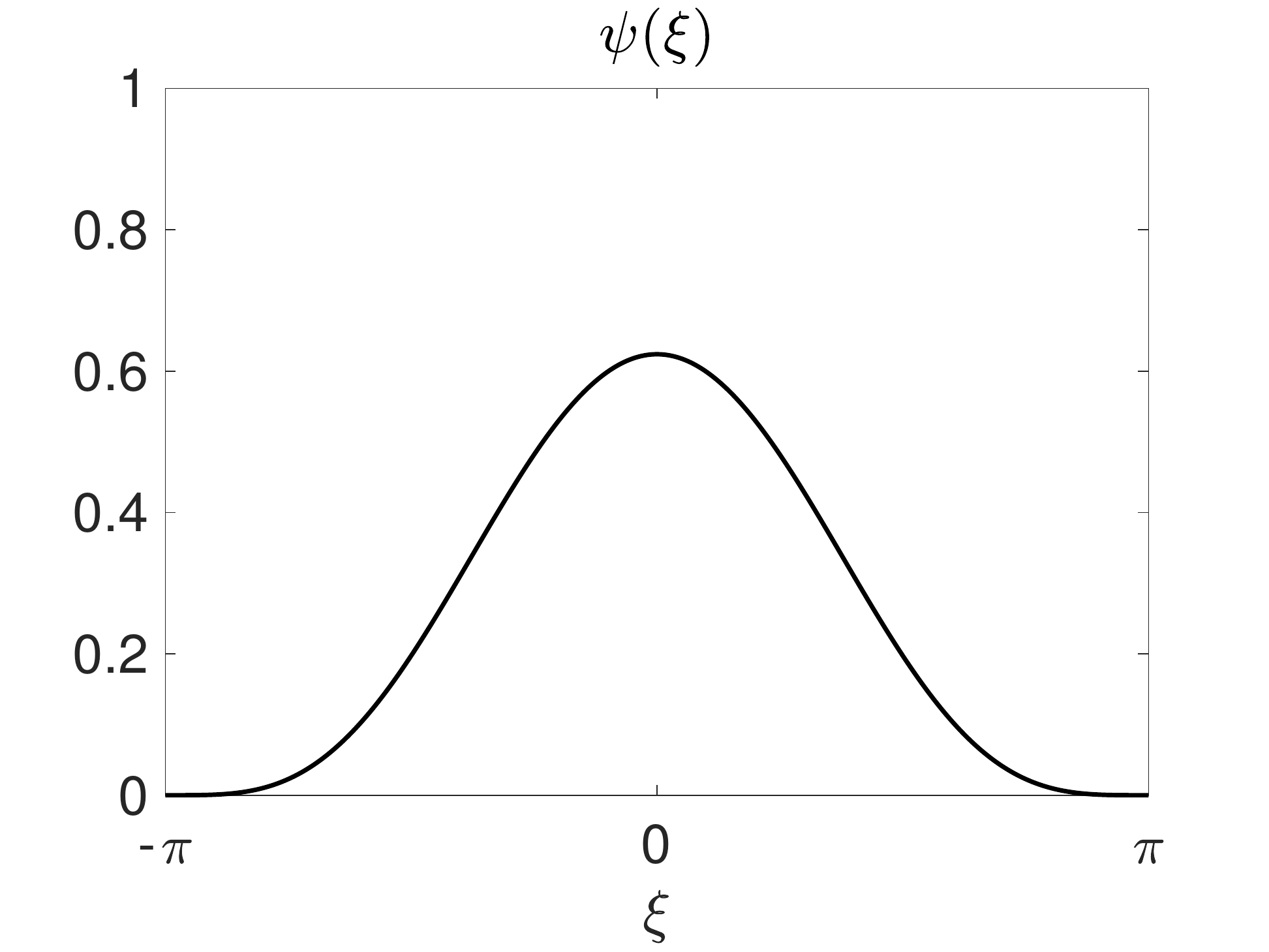}
\caption{Illustration of parameter $\alpha(x)$ (left) and wavelet $\psi(\xi)$ (right).}
\label{fig:alpha1D}
\end{figure}

For the numerical approximation, we consider an incoming wavelet of the form $u_I(x,t) = \omega\psi(\omega(t-x/c_0))$, with
\begin{align*}
\psi(\xi) &= \begin{cases}
\displaystyle\frac{(\xi-\pi)^4(\xi+\pi)^4}{3840\pi(21-2\pi^2)}, & \xi\in(-\pi,\pi), \\
0, & \text{otherwise}.
\end{cases}
\end{align*}
An illustration of $\psi$ is given in Figure \ref{fig:alpha1D}. 

At the boundary, we apply a perfectly matched layer of width $W=c_0\pi/\omega$, i.e. of width half a wave length. The equations for the absorbing boundary layer in 1D are the same as those in 2D given in \eqref{eq:PML}, but with $\zeta_2\equiv 0$ and with a scalar field $s(x)$ instead of $\vs(\vx)$. We choose the damping parameter $\zeta_1(x)=\zeta(x)$ as in \cite{cohen2002}, where
\begin{align}
\zeta(x) &= \begin{cases}
|\log(R)|\frac{3}{2W} \left(\frac{W-(1-|x|)}{W} \right)^2 , & |x| > 1 \\
0, &\text{otherwise},
\end{cases}
\end{align}
with $R=10^{-10}$ the expected artificial reflection. 

For the time step size, we choose $\Delta t=T_{up}/m$, with $m$ the smallest positive integer such that
\begin{align*}
\Delta t = \frac{T_{up}}{m} \leq c_{CFL} \frac{2}{\sqrt{\lambda^*_{\max}}},
\end{align*}
where the CFL number is chosen as $c_{CFL}=0.9$, and where $\lambda_{\max}^*$ is an upper bound of the largest eigenvalue $\lambda_{\max}$ of the discrete spatial differential operator $\beta^{-1}\nabla\cdot\alpha\nabla$, given by
\begin{align*}
\lambda^*_{\max} := \frac{1}{h_{K}^2} \frac{\alpha_{\max}}{\beta_{\min}} \sup_{u\in\hU\setminus\{0\}} \frac{(\nabla u,\nabla u)_{\{\he\}}}{(u,u)_{\{\he\}}},
\end{align*}
where $\alpha_{\max}:=\sup_{x\in\Omega}\alpha(x)$, and $\beta_{\min}:=\inf_{x\in\Omega} \beta(x)$. A smaller value of $T_{up}$ results in fewer neighbouring elements being marked for refinement, and therefore in fewer degrees of freedom on average. However, it also means that the mesh needs to be updated more frequently. In other words, choosing $T_{up}$ very small might slow down the method due to computational overhead, while choosing $T_{up}$ very large might render the method less efficient due to many additional neighbouring elements being refined. In the numerical examples, we choose $T_{up}$ as half a time-period, i.e. $T_{up}=\pi\omega^{-1}$.

For the spatial discretisation, we use quadratic elements (so degree $p=2$). We compute for the time interval $(t_0,t_{stop})=(t_0,T_{j_{stop}})$, where $t_0=-0.5-\pi\omega^{-1}$ and where $j_{stop}$ is determined using the stopping criterion in \eqref{eq:stop} with $\epsilon_0=\omega/100$. For the mesh refinement criterion in \eqref{eq:projErr}, we use $\eta_0=\omega/100$. An overview of the $L^2(\Omega_0)$ error $err_{2}:= \|U_S-U_h(\cdot,T_{j_{stop}})\|_{\Omega_0}$ and average number of degrees of freedom $\overline{n}_{DOF}:=\frac{1}{j_{stop}}\sum_{j=1}^{j_{stop}} |\cX(\cT_j)|$ is given in Table \ref{tab:err1a} for the adapted finite element method. The results are compared with 
the classical finite element method using a uniform mesh of width
$h=h_K$ and using the same polynomial degree, time step size, and
stopping time as for the adapted finite element method. From this
table, we can see that the error of the adapted finite element method
and classical finite element method behave very similarly, whereas the average number of degrees of freedom grows at a significantly slower rate for the adapted finite element method as the frequency $\omega$ increases. In particular, Tables \ref{tab:err1a} and \ref{tab:rate} illustrate that the average number of degrees of freedom is almost independent of $\omega$ for the adaptive finite element method in 1D.

\begin{table}[h]
\centering
{\tabulinesep=0.5mm
\begin{tabu}{|c| c c c c r r| r r|} \hline
& \multicolumn{6}{c|}{AFEM} & \multicolumn{2}{c|}{FEM} \\ 
$\omega$ & $\{h_k\}$ & $T_{up}$ & $m$ & $j_{stop}$ & $err_{2}$ & $\overline{n}_{DOF}$ & $err_{2}$ & $\overline{n}_{DOF}$ \\ \hline
$10\pi$ & $\{\frac{1}{5},\frac{1}{50}\}$ & $\frac{1}{10}$ & 28 & 20 & 1.40e-02 & 1.14e+02 & 1.00e-02 & 2.21e+02  \\
$20\pi$ & $\{\frac{1}{5},\frac{1}{10},\frac{1}{100}\}$ & $\frac{1}{20}$ & 28 & 36 & 1.67e-02 & 1.43e+02 & 1.71e-02 & 4.21e+02 \\
$40\pi$ & $\{\frac{1}{5},\frac{1}{20},\frac{1}{200}\}$ & $\frac{1}{40}$ & 28 & 63 & 3.92e-02 & 1.78e+02 & 3.93e-02 & 8.21e+02 \\ 
$80\pi$ & $\{\frac{1}{5},\frac{1}{40},\frac{1}{400}\}$ & $\frac{1}{80}$ & 28 & 123 & 6.49e-02 & 2.12e+02 & 5.92e-02 & 1.62e+03 \\ \hline
\end{tabu}
}
\caption{Estimated $L^2(\Omega_0)$ error and average number of degrees of freedom for the quadratic adapted (AFEM) and classical (FEM) finite element approximation to the 1D Helmholtz problem for different angular frequencies $\omega$. To estimate the error, we take the numerical approximation on a uniform mesh of width $h_K/2$ as reference solution.}
\label{tab:err1a}
\end{table}

To compute the errors in Table \ref{tab:err1a}, we use the discrete solution obtained on a uniform mesh of width $h_{K/2}$ as reference solution. However, this error does not take into account the error produced by the absorbing boundary layer or the truncation of the wave field at time $t_{stop}$. To measure these errors, we compute the following numerical approximations and reference solutions.
\begin{itemize}
  \item $U_h$: the adapted finite element approximation considered in Table \ref{tab:err1a}.
  \item $U^1$: the reference solution used in Table \ref{tab:err1a}.  
  \item $U^2$: similar to $U^1$, but using the exact absorbing boundary condition $\partial_tu+c_0\partial_xu = 0$ on $\partial\Omega_0$ instead of an absorbing boundary layer.
  \item $U^3$: similar to $U^2$, but using a time interval $(t_0,t_0+100)$ instead of $(t_0,t_{stop})$ ($t_{stop}-t_0<2.5$ for all cases in Table \ref{tab:err1a}).
\end{itemize}
The difference $U_h-U^1$ indicates the error due to the spatial and time discretisation, $U^1-U^2$ indicates the error due to the absorbing boundary layer, and $U^2-U^3$ indicates the error due to the truncation in time. An overview of these errors is given in Table \ref{tab:err1b}. From this table, we can see that the error is dominated by the discretisation error, whereas the errors due to the absorbing boundary layer and truncation in time are negligible. We will use this as a motivation to also estimate the error by $\|U_h-U^1\|_{\Omega_0}$ in the 2D case. 

\begin{table}[h]
\centering
{\tabulinesep=0.5mm
\begin{tabu}{|c| r r r|} \hline
$\omega$ & $\|U_h-U^1\|_{\Omega_0}$ & $\|U^1-U^2\|_{\Omega_0}$ & $\|U^2-U^3\|_{\Omega_0}$  \\ \hline
$10\pi$ & 1.40e-02 & 2.42e-04 & 4.50e-03 \\
$20\pi$ & 1.67e-02 & 2.00e-04 & 3.45e-03 \\
$40\pi$ & 3.92e-02 & 1.38e-04 & 3.94e-03 \\ 
$80\pi$ & 6.49e-02 & 1.21e-04 & 2.12e-03 \\ \hline
\end{tabu}
}
\caption{Estimated $L^2(\Omega_0)$ error due to the spatial and time discretisation $\|U_h-U^1\|_{\Omega_0}$, due to the absorbing boundary layer $\|U^1-U^2\|_{\Omega_0}$, and due to the truncation in time $\|U^2-U^3\|_{\Omega_0}$ for the 1D test cases.}
\label{tab:err1b}
\end{table}

An illustration of $u_h$ for the case $\omega=40\pi$ and an illustration of the corresponding space-time mesh as described in Section \ref{sec:FT} are given in Figure \ref{fig:spaceTimeMesh}. 

\begin{figure}
\begin{center}
\includegraphics[width=0.45\textwidth]{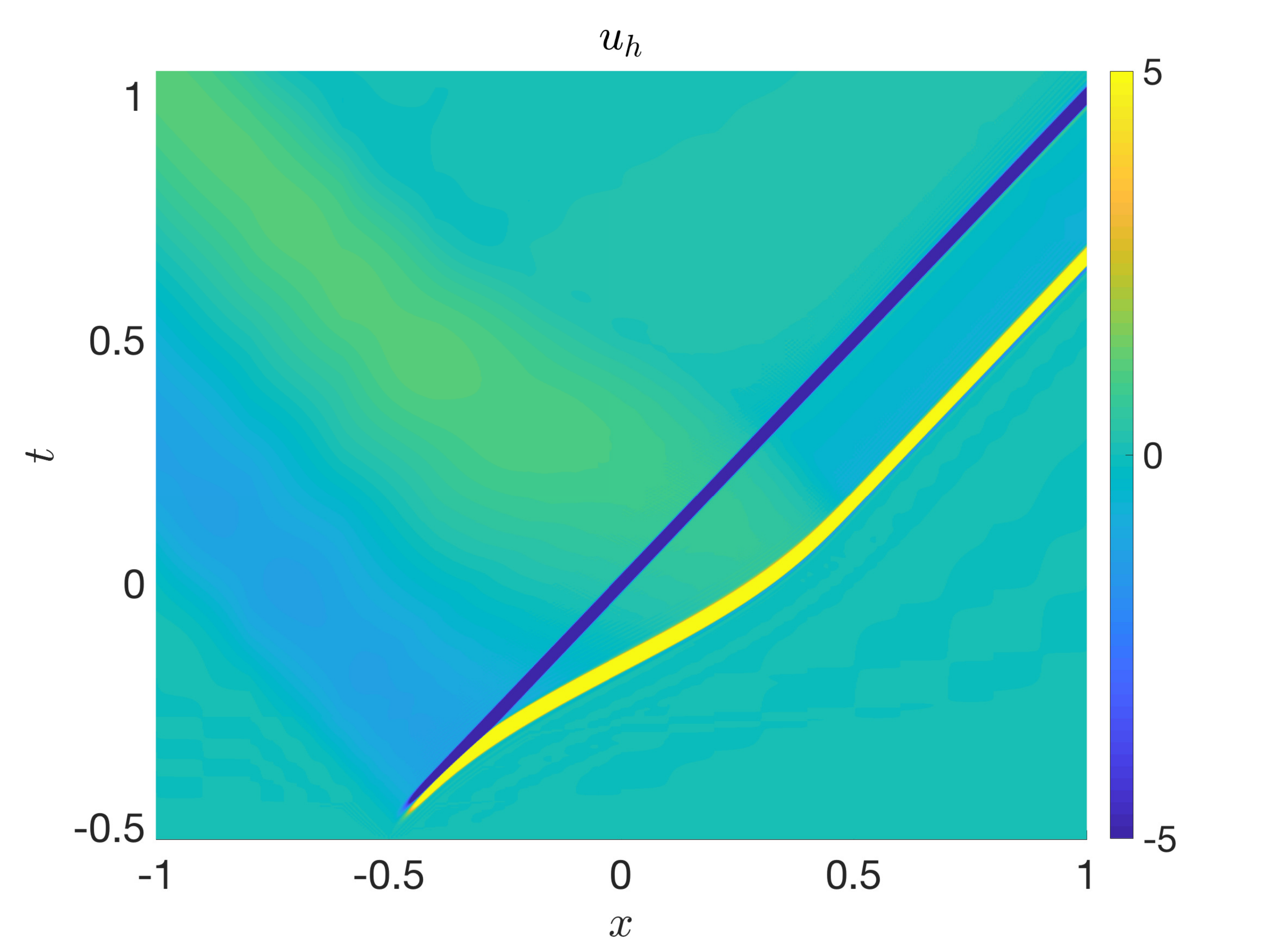}
\includegraphics[width=0.45\textwidth]{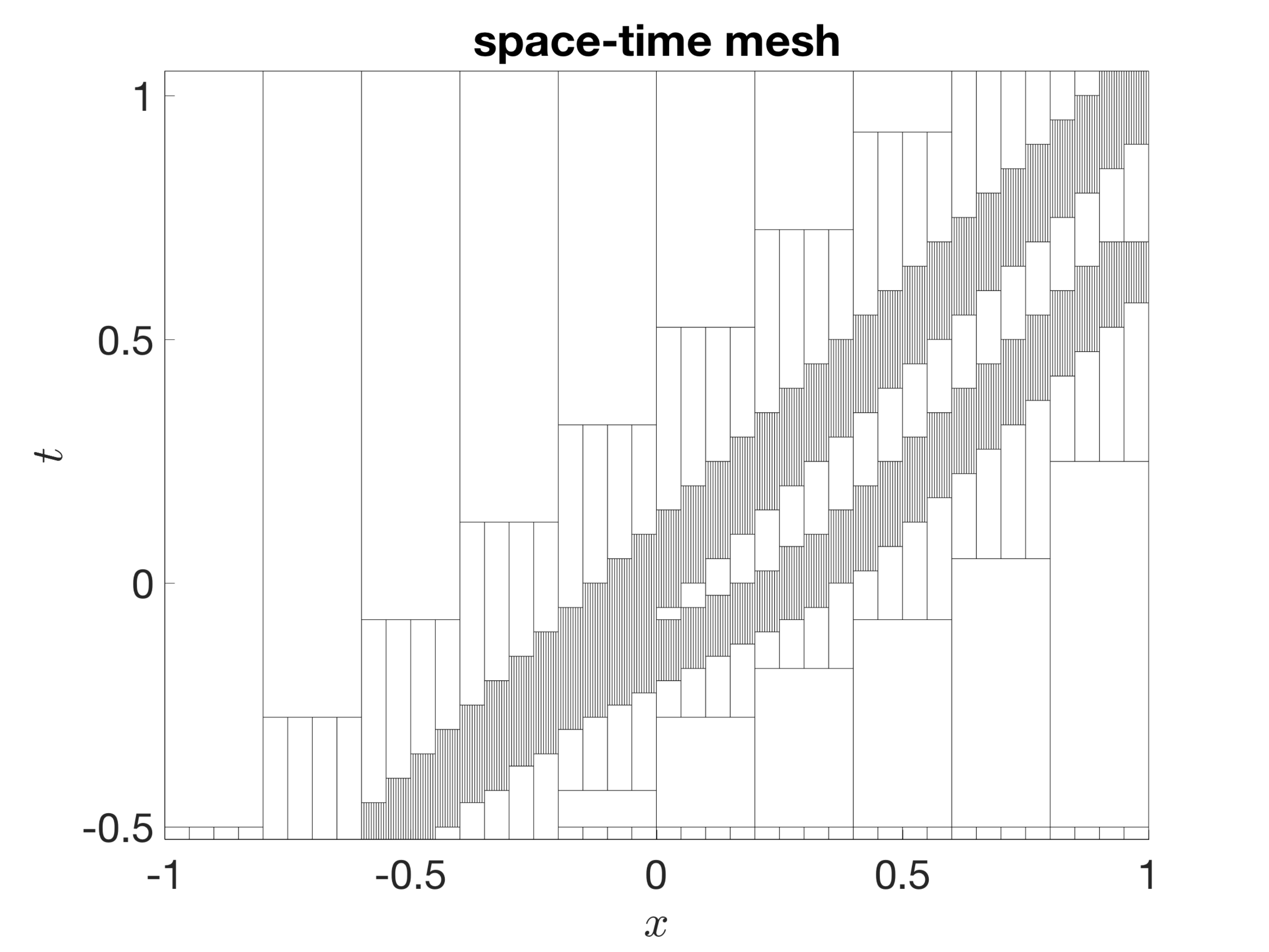}
\end{center}
\caption{Illustration of $u_h$ and the space-time mesh described in Section \ref{sec:FT} for the 1D test case with $\omega=40\pi$. In the actual algorithm, we never explicitly compute the space-time mesh.}
\label{fig:spaceTimeMesh}
\end{figure}

\subsection{Numerical example in 2D: incoming plane wave}
\label{sec:num2Dplane}
For the first 2D example, we consider a domain $\Omega_0=(-1,1)^2$, with spatial parameters
\begin{align*}
\alpha(x_1,x_2) &= \begin{cases}
1 + 3(1-2\sqrt{x_1^2+x_2^2})^2(1+2\sqrt{x_1^2+x_2^2})^2, & \sqrt{x_1^2+x_2^2} \leq \frac12, \\
1, & \text{otherwise},
\end{cases} \\
\beta(x_1,x_2) &= 1,
\end{align*}
and an incoming plane wave $U_I(x_1,x_2) = e^{\im \omega x_1/c_0}$, with $c_0=1$. An illustration of $\alpha$ is given in Figure \ref{fig:alpha2D}.

\begin{figure}[h]
\centering
\includegraphics[width=0.45\textwidth]{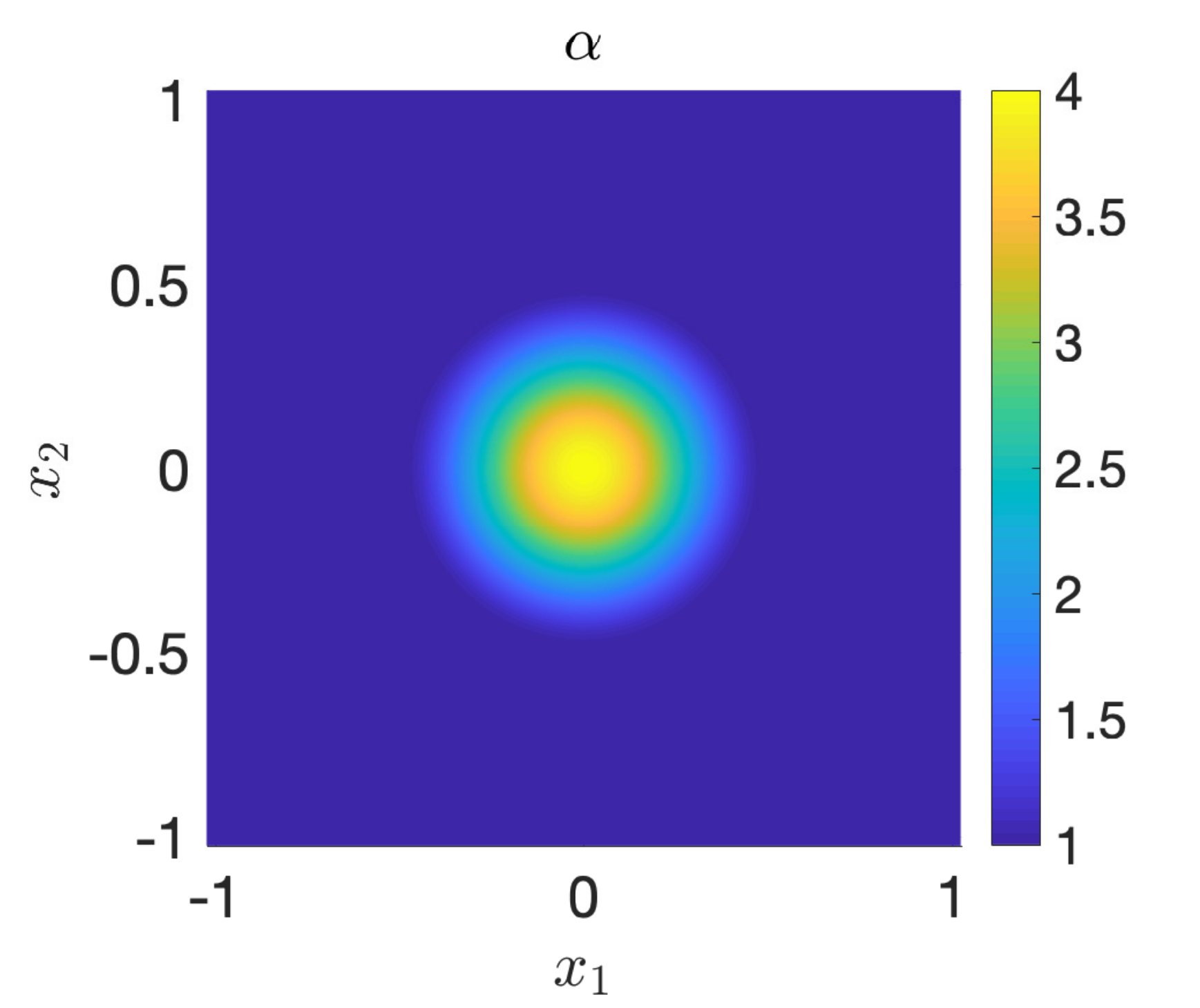}
\caption{}
\label{fig:alpha2D}
\end{figure}

For the numerical approximation, we consider an incoming wavelet of the form $u_I(x_1,x_2,t) = \omega\psi(\omega(t-x_1/c_0))$, with $\psi$ as in the 1D example. At the boundary, we apply a perfectly matched layer of width $W=c_0\pi/\omega$, i.e. of width half a wave length. We choose the damping parameters $\zeta_1(x_1)=\zeta(x_1)$ and $\zeta_2(x_2)=\zeta(x_2)$, with $\zeta(x)$ as in the 1D example.

We compute for the time interval $(t_0, t_{stop}) = (t_0,
T_{j_{stop}})$, where $t_0 = -0.5-\pi\omega^{-1}$. For the spatial
discretisation, we use biquadratic elements. The time step size
$\Delta t$ and the parameters $\eta_0$ and $\epsilon_0$ are chosen in
the same way as in the 1D case. We compute the $L^2(\Omega_0)$ error
and compare the results with a classical finite element method with a
uniform mesh in the same way as we did for the 1D case. The results
are presented in Table \ref{tab:err2a}. From this table, we can see
that, just as in the 1D case, the errors of the adaptive finite
element method and classical finite element method behave similarly, whereas the average number of degrees of freedom grows at a significantly slower rate for the adapted finite element method as $\omega$ increases. In particular, Table \ref{tab:rate} illustrates that for the adaptive finite element method, the average number of degrees of freedom grows almost linearly with $\omega$ instead of as $\omega^2$.

\begin{table}[h]
\centering
{\tabulinesep=0.5mm
\begin{tabu}{|c| c c c c r r| r r|} \hline
& \multicolumn{6}{c|}{AFEM} & \multicolumn{2}{c|}{FEM} \\ 
$\omega$ & $\{h_k\}$ & $T_{up}$ & $m$ & $j_{stop}$ & $err_{2}$ & $\overline{n}_{DOF}$ & $err_{2}$ & $\overline{n}_{DOF}$ \\ \hline
$10\pi$ & $\{\frac{1}{5},\frac{1}{50}\}$ & $\frac{1}{10}$ & $39$ & $22$ & 7.47e-03 & 1.93e+04 & 4.97e-03 & 4.88e+04 \\
$20\pi$ & $\{\frac{1}{5},\frac{1}{10},\frac{1}{100}\}$ & $\frac{1}{20}$ & $39$ & $39$ & 8.71e-03 & 4.15e+04 & 8.43e-03 & 1.77e+05 \\
$40\pi$ & $\{\frac{1}{5},\frac{1}{20},\frac{1}{200}\}$ & $\frac{1}{40}$ & $39$ & $73$ & 1.77e-02 & 9.33e+04 & 1.74e-02 & 6.74e+05 \\ 
$80\pi$ & $\{\frac{1}{5},\frac{1}{40},\frac{1}{400}\}$ & $\frac{1}{80}$ & $39$ & $139$ & 3.48e-02 & 2.09e+05 & 3.44e-02 & 2.63e+06 \\ 
\hline
\end{tabu}
}
\caption{Estimated $L^2(\Omega_0)$ error and average number of degrees of freedom for the biquadratic adapted (AFEM) and classical (FEM) finite element approximation to the 2D plane wave Helmholtz problem for different angular frequencies $\omega$. To estimate the error, we take the numerical approximation on a uniform mesh of width $h_K/2$ as the exact solution. }
\label{tab:err2a}
\end{table}

\begin{table}[h]
\centering
{\tabulinesep=0.5mm
\begin{tabu}{|c| r r r| r r r|} \hline
&\multicolumn{3}{c|}{1D} & \multicolumn{3}{c|}{2D} \\
$\omega$ & $\overline{n}_{DOF}$ & ratio & rate & $\overline{n}_{DOF}$ & ratio & rate \\ \hline
$10\pi$ & 1.14e+02 &  &  & 1.93e+04 &  &  \\
$20\pi$ & 1.43e+02 & 1.25 & 0.33 & 4.15e+04 & 2.15 & 1.10 \\
$40\pi$ & 1.78e+02 & 1.24 & 0.31 & 9.33e+04 & 2.25 & 1.17 \\ 
$80\pi$ & 2.12e+02 & 1.19 & 0.25 & 2.09e+05 & 2.24 & 1.16 \\ \hline
\end{tabu}
}
\caption{Estimated growth rate of the average number of degrees of freedom $\overline{n}_{DOF}$ with respect to the frequency $\omega$ for the adaptive finite element method. The results correspond to the incoming plane wave problem in 1D and 2D.}
\label{tab:rate}
\end{table}

Snapshots of the time-dependent wave field for different frequencies are shown in Figure \ref{fig:uh}. This figure shows that the wavefront gets sharper as $\omega$ increases, while away from the wavefront, e.g. on the left of $x_1=0$, the wave field is similar for different frequencies. An illustration of the total time-harmonic field and the error for the case $\omega=40\pi$ is given in Figure \ref{fig:Uh}.

\begin{figure}[h]
\begin{center}
\includegraphics[width=0.37\textwidth]{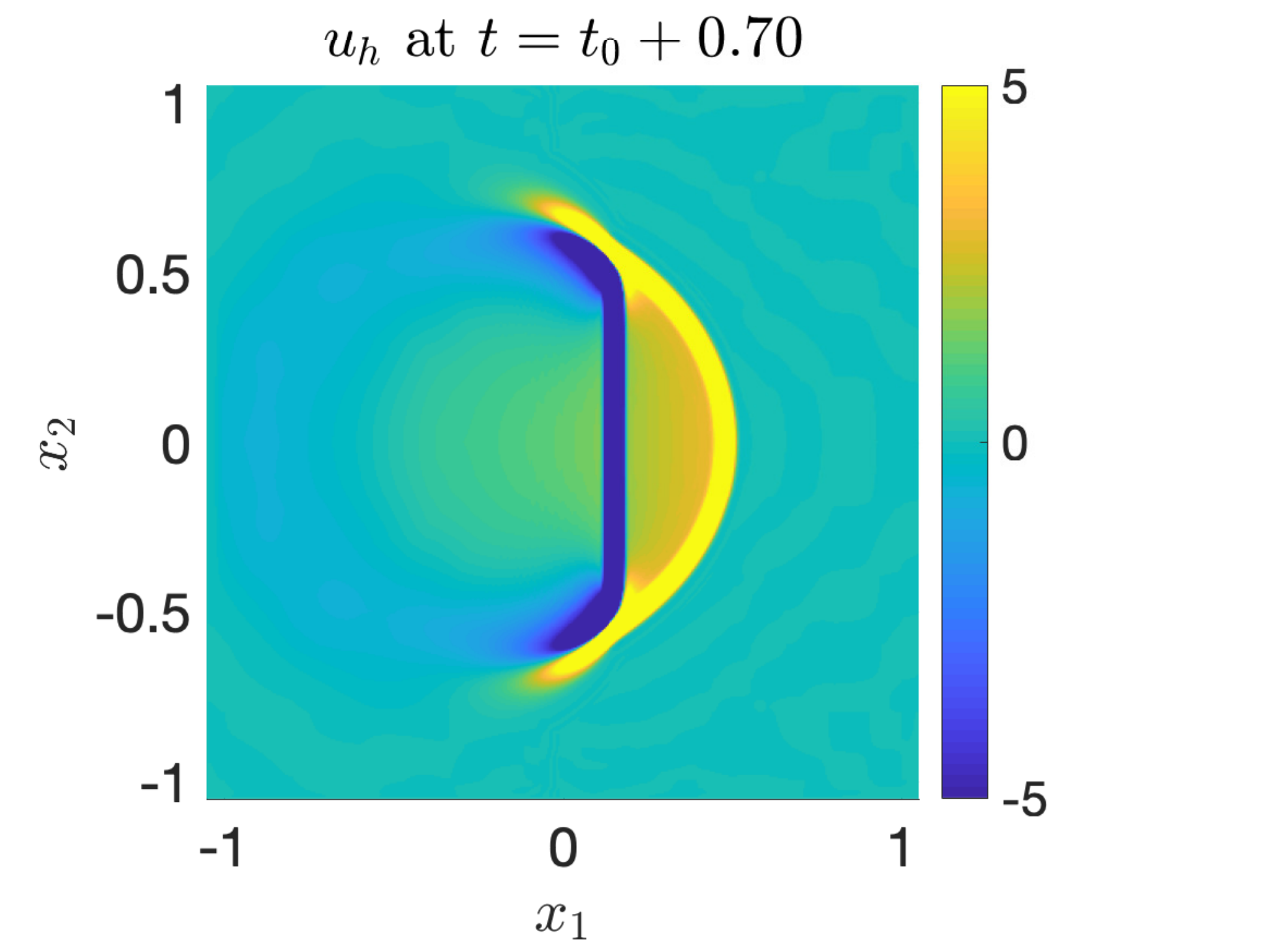}\hspace{-1cm}
\includegraphics[width=0.37\textwidth]{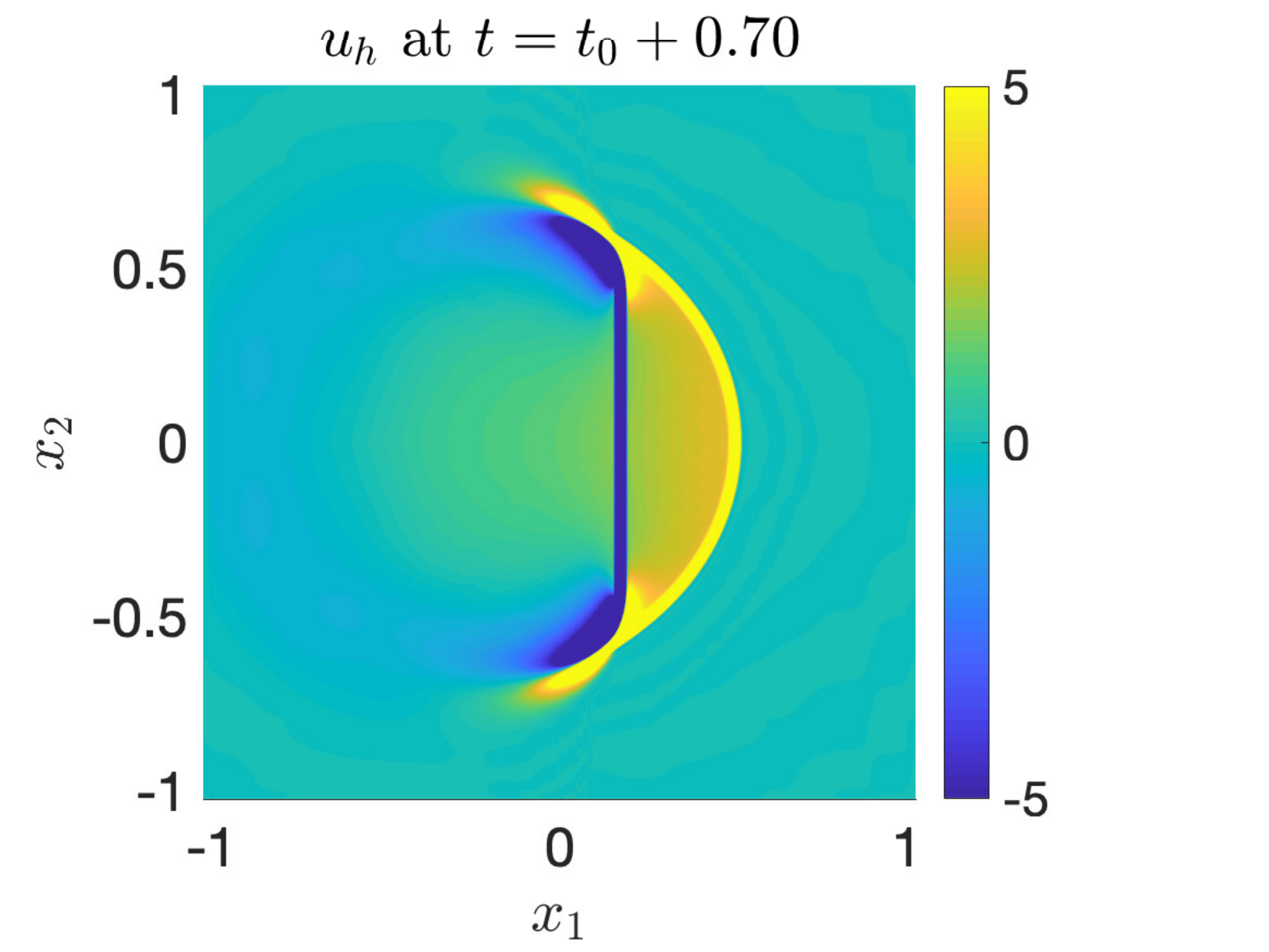}\hspace{-1cm}
\includegraphics[width=0.37\textwidth]{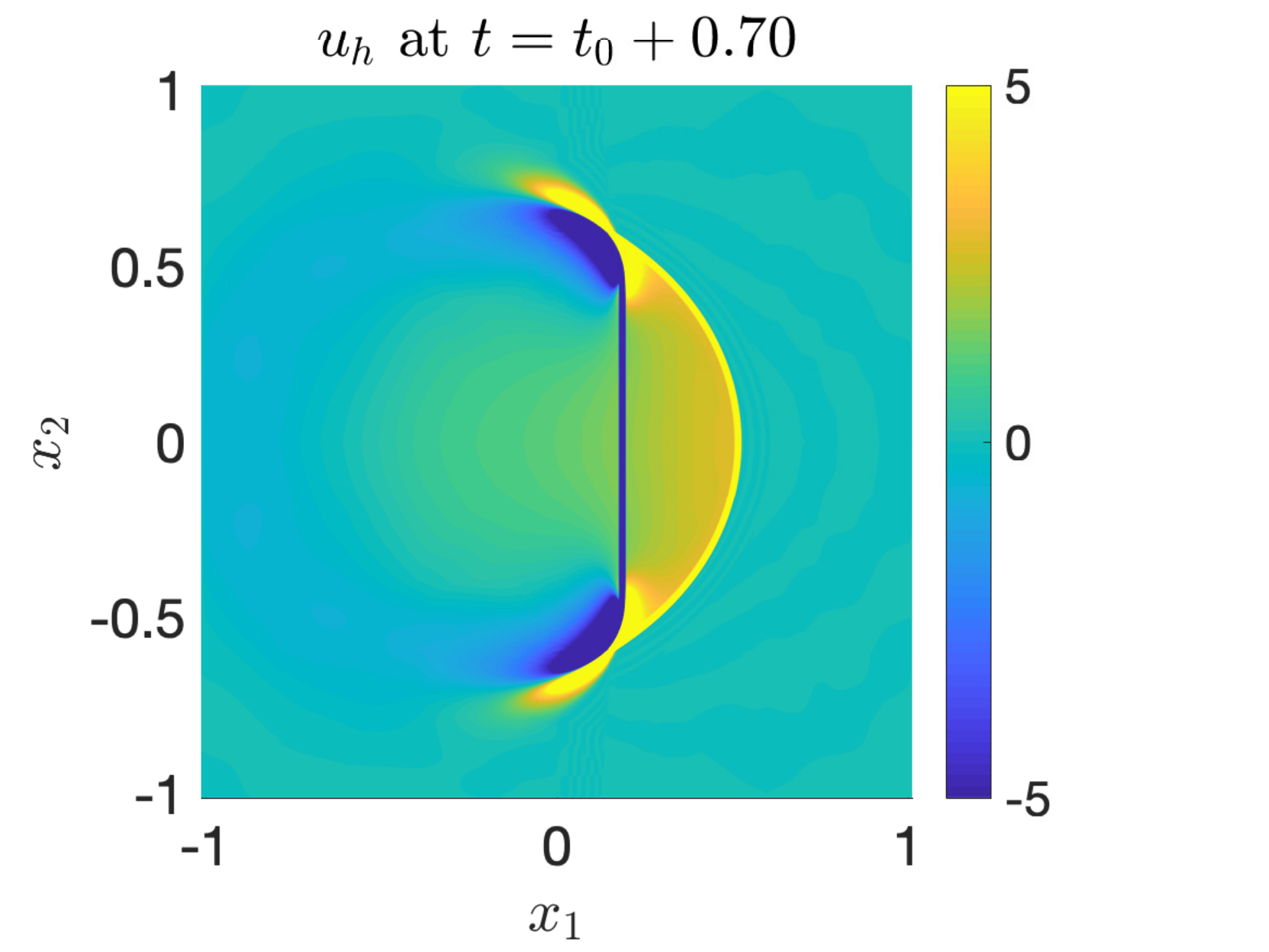}
\end{center}
\caption{Snapshot of $u_h$ at time $t=t_0+0.70$ for the 2D plane wave test case with $\omega=20\pi$ (left), $\omega=40\pi$ (middle), and $\omega=80\pi$ (right).}
\label{fig:uh}
\end{figure}

\begin{figure}[h]
  \begin{center}
    \begin{minipage}[h]{0.48\textwidth}
      \includegraphics[width=\textwidth]{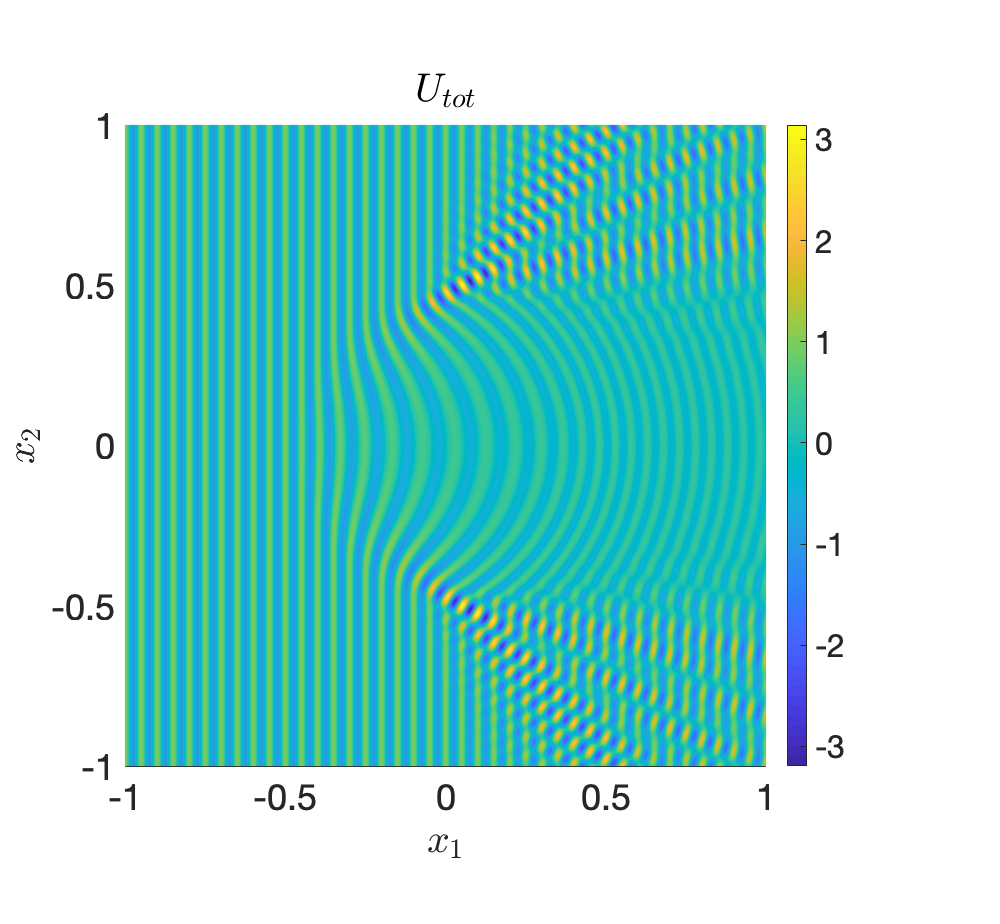}
    \end{minipage}
    \begin{minipage}[h]{0.455\textwidth}
      \includegraphics[width=\textwidth]{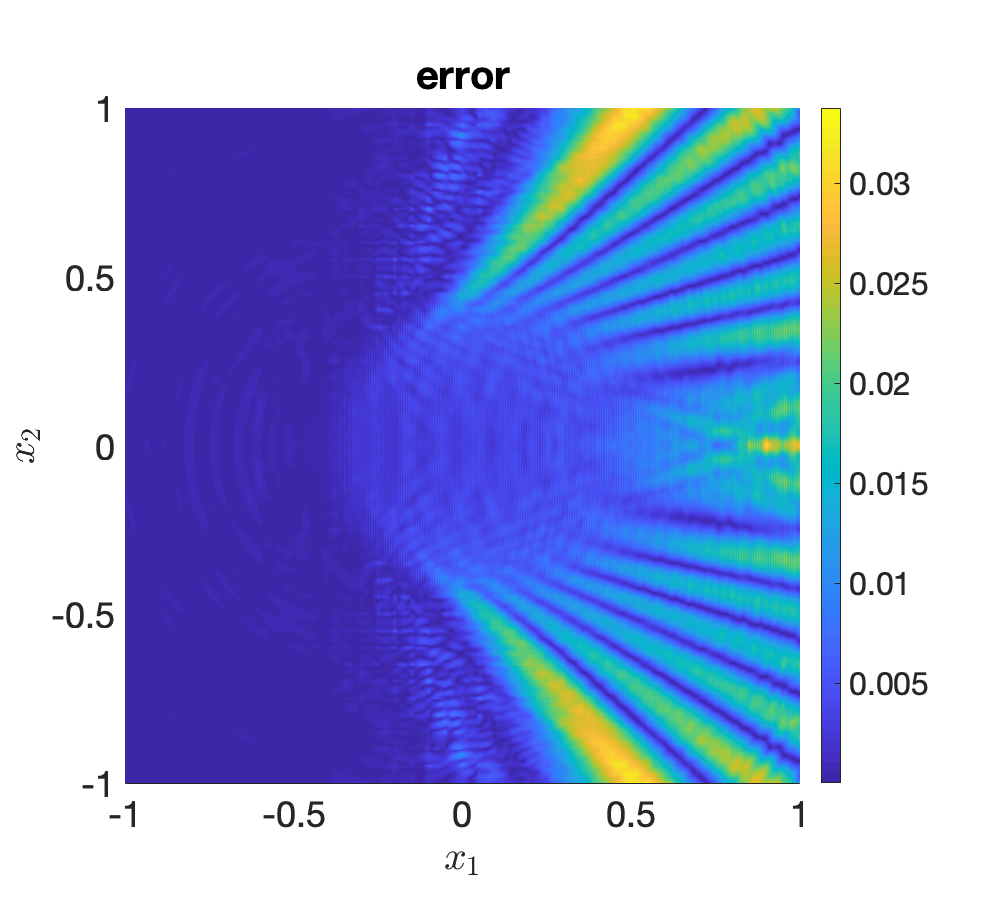}
      \end{minipage}
\end{center}
\caption{Total time-harmonic field $U_{tot}=U_h+U_I$ (left) and error $|U_S-U_h|$ (right) for the 2D plane wave test case with $\omega=40\pi$.}
\label{fig:Uh}
\end{figure}

To illustrate the adaptive mesh refinement procedure, we define, for each set of parent elements $\cP$, the function $\Call{level}{\cP}:\Omega
\rightarrow\mR$ as
\begin{align*}
\Call{level}{\cP}(\vx) &:= \max(\{0\}\cup\{k \;|\; \vx\in\e \text{ for some }\e\in\cP\cap\cT^k\}).
\end{align*}
In other words, $\Call{level}{\cP}(\vx)$ returns the level of the finest element in $\cP$ that contains $\vx$. If no element in $\cP$ contains $\vx$, then $\Call{level}{\cP}(\vx)$ returns $0$. An illustration of the mesh adaptation algorithm is given in Figure \ref{fig:meshUpdate}.

\begin{figure}[h]
\begin{center}
\includegraphics[width=0.45\textwidth]{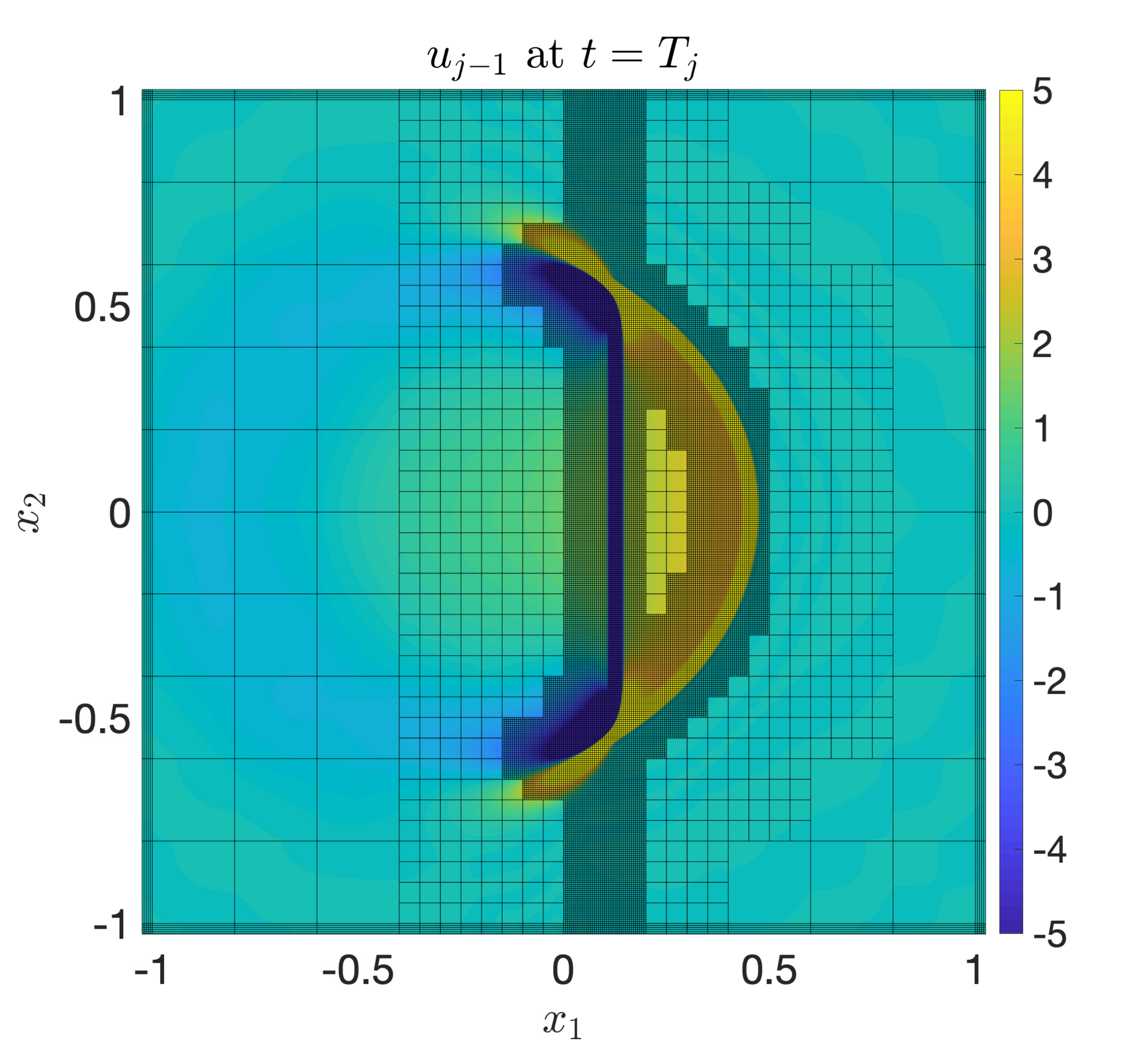}
\includegraphics[width=0.45\textwidth]{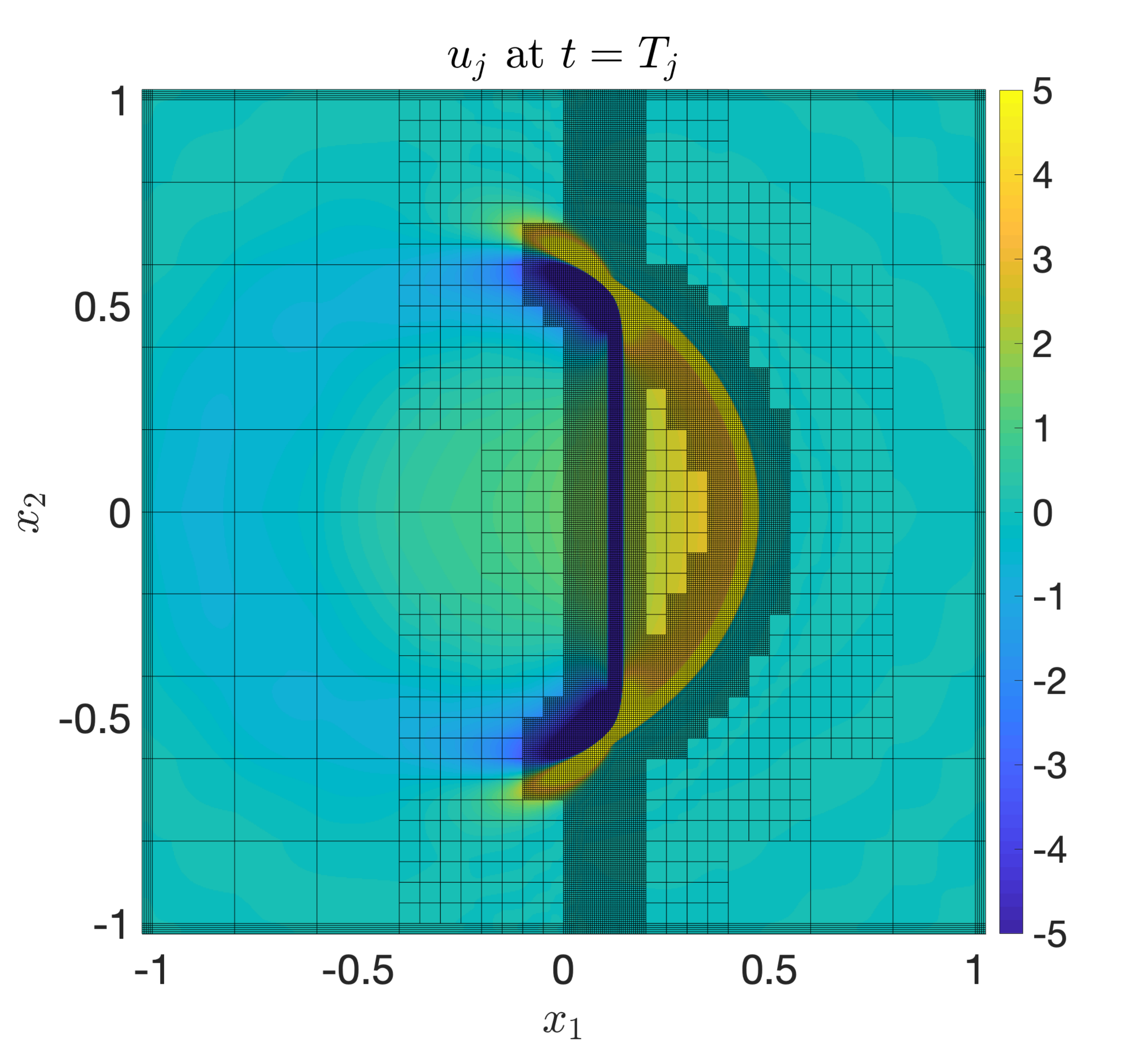} \\
\includegraphics[width=0.32\textwidth]{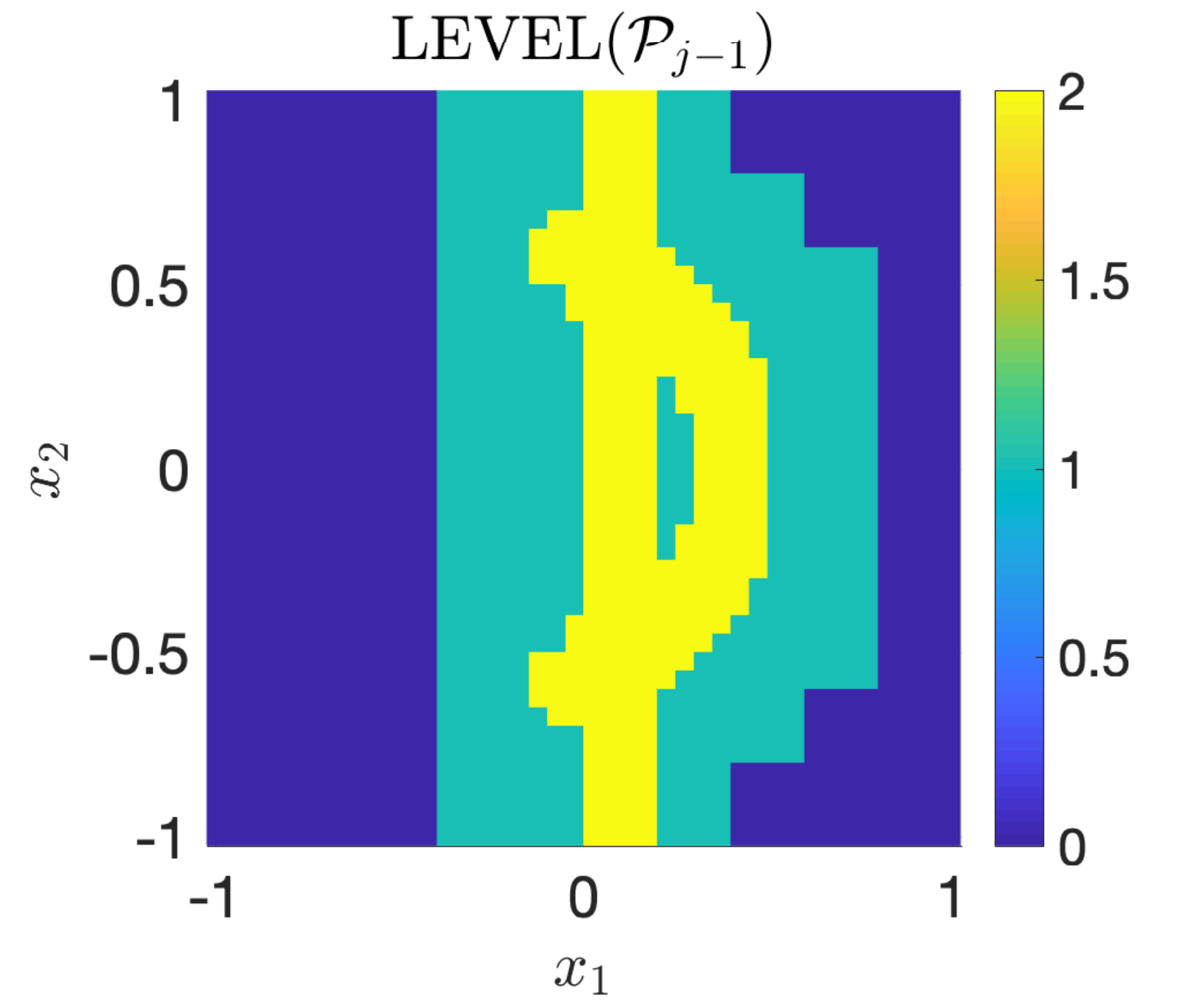}
\includegraphics[width=0.32\textwidth]{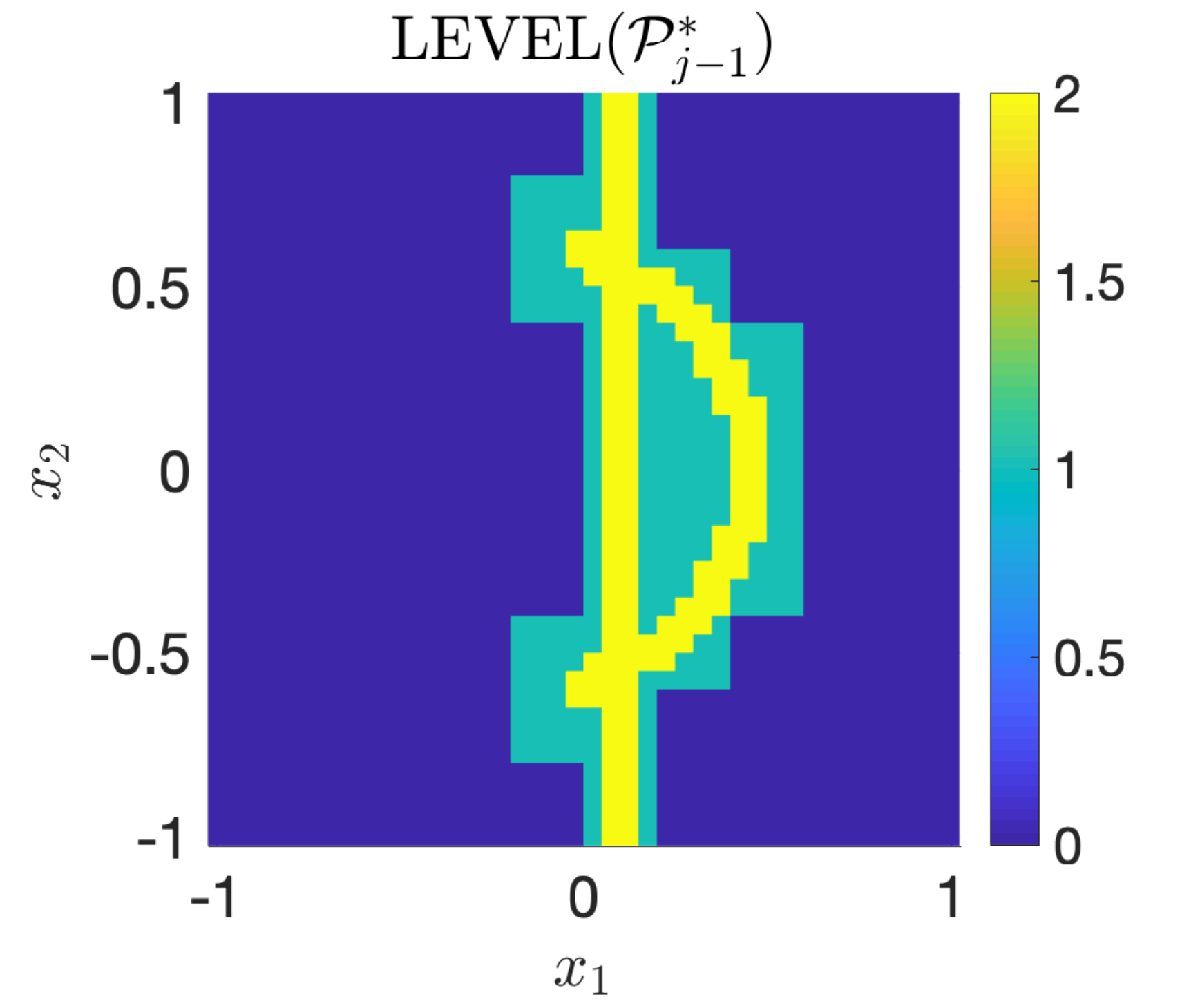}
\includegraphics[width=0.32\textwidth]{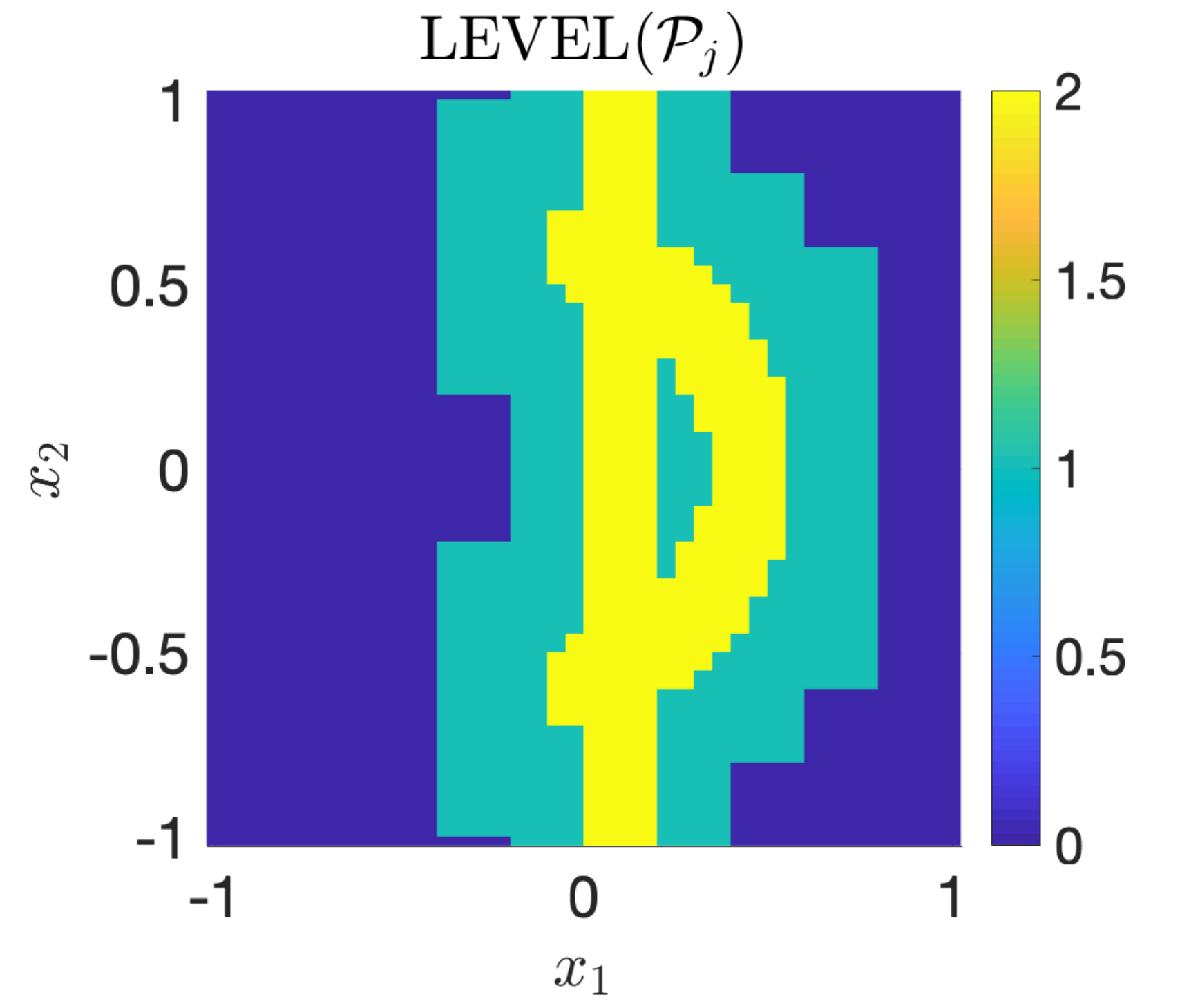}
\end{center}
\caption{Top: illustration of wave field $u_{j-1}$ and the updated wave field $u_j$ at time $T_j=t_0+0.65$, for the 2D plane wave test case with $\omega=40\pi$ using three levels of nested meshes with $\{h_k\}=\{\frac{1}{5},\frac{1}{20},\frac{1}{200}\}$. Bottom: parent elements $\textsc{level}(\cP_{j-1})$, marked elements $\textsc{level}(\cP^{*}_{j-1})$, and marked neighbour elements $\textsc{level}(\cP_j)$. Here, $\cP_{j-1}$, $\cP^{*}_{j-1}$, and $\cP_j$ are defined as in Algorithm \ref{alg:updateMesh}.}
\label{fig:meshUpdate}
\end{figure}

\subsection{Numerical example in 2D: point source}
\label{sec:num2Dpoint}
As a next 2D numerical example, we consider the Helmholtz equation of the form in \eqref{eq:Helmholtz2} with a slightly smeared out point source $F$ of the form
\begin{align*}
F(x_1,x_2) &:= \frac{200}{\lambda^2} f_0\left(\frac{\rho(x_1,x_2)}{\frac12\lambda}\right) \qquad
f_0(\xi) := \begin{cases}
(\xi^2-1)^4, &|\xi| \leq 1, \\
0, & \text{otherwise},
\end{cases}
\end{align*}
where 
$\lambda=2\pi/k_0$ denotes the wave length in the exterior domain, 
$k_0=\omega/c_0$ denotes the wave number in the exterior domain, $\rho(x_1,x_2) = \sqrt{(x_1-x_{1,0})^2 + (x_2-x_{2,0})^2}$ denotes the distance to $(x_{1,0},x_{2,0})$, and $(x_{1,0}, x_{2,0}):=(0.5, 0.5)$ is the position of the point source. Note that the support of $F$ is centered at $(x_{1,0},x_{2,0})$ and has a diameter of one wave length $\lambda$. The domain, absorbing boundary layer, and spatial parameters $\alpha$ and $\beta$ are chosen as in the previous example.

For the numerical approximation, we use the same spatial discretisation and parameters $\eta_0$ and $\epsilon_0$ as in the previous example. The initial time is $t_0=-\pi\omega^{-1}$. Similar to the previous examples, we compare the numerical results of the adaptive method with the classical finite element method for different frequencies. The results are presented in Table \ref{tab:err2b}. Again, the accuracy of the adaptive and classical method are comparable, whereas the average number of degrees of freedom is significantly smaller for the adaptive method and grows almost linearly with $\omega$ instead of quadratically.

\begin{table}[h]
\centering
{\tabulinesep=0.5mm
\begin{tabu}{|c| c c c c r r| r r|} \hline
& \multicolumn{6}{c|}{AFEM} & \multicolumn{2}{c|}{FEM} \\ 
$\omega$ & $\{h_k\}$ & $T_{up}$ & $m$ & $j_{stop}$ & $err_{2}$ & $\overline{n}_{DOF}$ & $err_{2}$ & $\overline{n}_{DOF}$ \\ \hline
$10\pi$ & $\{\frac{1}{5},\frac{1}{50}\}$ & $\frac{1}{10}$ & $39$ & $55$ & 1.26e-02 & 9.79e+03 & 8.07e-03 & 4.88e+04   \\
$20\pi$ & $\{\frac{1}{5},\frac{1}{10},\frac{1}{100}\}$ & $\frac{1}{20}$ & $39$ & $65$ & 1.16e-02 & 2.81e+04 & 1.01e-02 & 1.77e+05  \\
$40\pi$ & $\{\frac{1}{5},\frac{1}{20},\frac{1}{200}\}$ & $\frac{1}{40}$ & $39$ & $87$ & 1.55e-02 & 7.89e+04 & 1.40e-02 & 6.74e+05  \\ 
$80\pi$ & $\{\frac{1}{5},\frac{1}{40},\frac{1}{400}\}$ & $\frac{1}{80}$ & $39$ & $147$ & 2.14e-02 & 1.82e+05 & 1.96e-02 & 2.63e+06  \\
\hline
\end{tabu}
}
\caption{Estimated $L^2(\Omega_0)$ error and average number of degrees of freedom for the biquadratic adapted (AFEM) and classical (FEM) finite element approximation to the 2D point source Helmholtz problem for different angular frequencies $\omega$. To estimate the error, we take the numerical approximation on a uniform mesh of width $h_K/2$ as the exact solution. }
\label{tab:err2b}
\end{table}

An illustration of the adaptive method is given in Figure \ref{fig:pointSource}. The left image of Figure \ref{fig:pointSource} shows a sharp circular wave front generated by the point source and an adapted mesh that is only refined near this wave front. The right image shows the approximated time-harmonic wave field.

\begin{figure}[h]
\begin{center}
\includegraphics[width=0.45\textwidth]{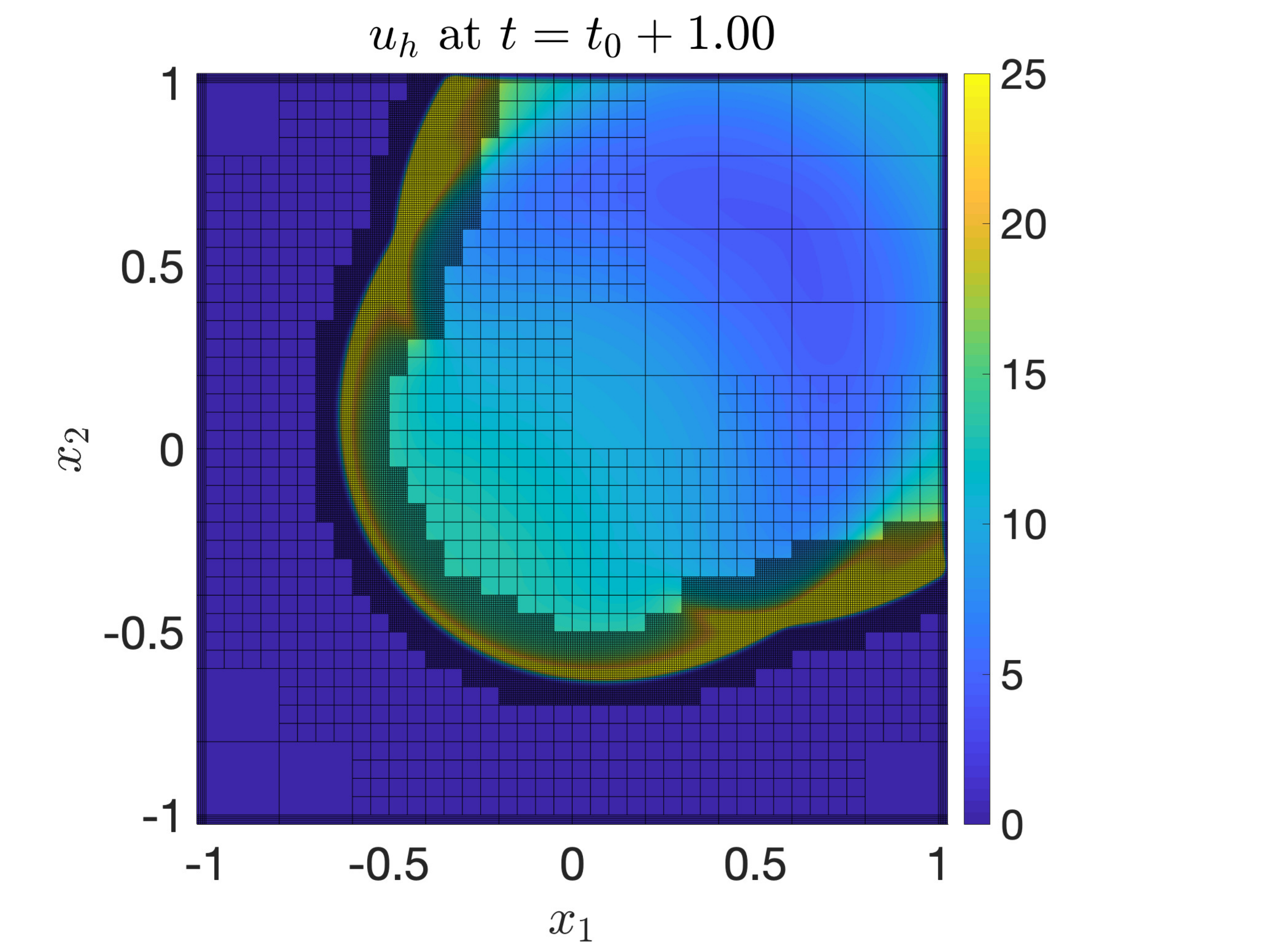} 
\includegraphics[width=0.45\textwidth]{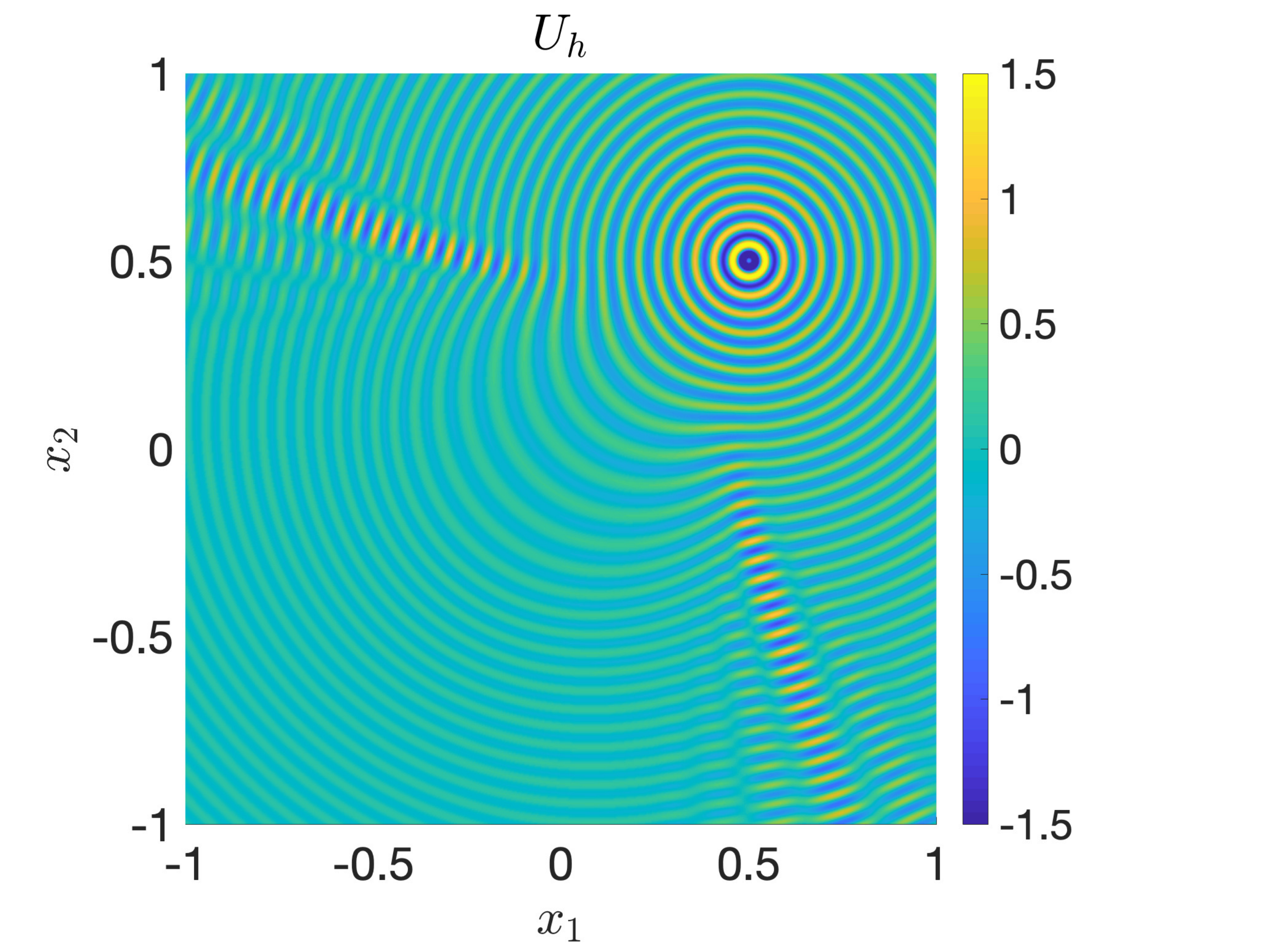} 
\end{center}
\caption{Snapshot of the wave field $u_h$ at time $T_j=t_0+1$ (left) and time-harmonic field $U_h$ (right) for the 2D point source test case with $\omega=40\pi$.}
\label{fig:pointSource}
\end{figure}

\subsection{Numerical example in 2D: trapping mode}
As a final 2D numerical example, we consider a scattering problem with a sound-soft scatterer that can trap waves. The scatterer is illustrated in Figure \ref{fig:trappingMode}. The domain, absorbing boundary layer, and incoming wave are chosen as in Section \ref{sec:num2Dplane} and the spatial parameters are given by $\alpha\equiv 1$, $\beta\equiv 1$.


For the numerical test, we only consider the case $\omega=30\pi$. We
use two nested meshes with mesh size
$\{h_k\}=\{\frac1{10},\frac1{150}\}$ and with biquadratic elements, a
mesh update time $T_{up}=1/30$, a mesh refinement criterion
$\xi_0=\frac{1}{100}\omega$ and a stopping criterion
$\epsilon_0=\frac{5}{100}\omega$. The initial time is
$t_0=0.4-\pi\omega^{-1}$, the number of time steps between each mesh
update is $m=20$, and the stopping criterion is triggered at
$t_{stop}=T_{j_{stop}}$, with $j_{stop}=332$. We compare the results
of the adaptive finite element method and classical finite element
method as in the previous examples. The results are listed in Table
\ref{tab:err2c}. An illustration of the adapted mesh and of the computed
time-harmonic wave field is also given in Figure
\ref{fig:trappingMode}. Due to trapping, it takes much longer before
the wave field vanishes, which makes time-domain approaches less
efficient. Furthermore, as illustrated in Figure
\ref{fig:trappingMode}, a large region around the trapping area requires a fine mesh, which makes the adaptive method less efficient. The average number of degrees of freedom is still smaller than for the classical finite element method, since the adaptive method correctly determines in which part of the domain the wave field is active.



\begin{table}[h]
\centering
{\tabulinesep=0.5mm
\begin{tabu}{|c| r r|} \hline
$\omega=30\pi$ & $err_2$ &$\overline{n}_{DOF}$ \\ \hline
AFEM & 1.77e-01 & 9.95e+04 \\
FEM & 1.57e-01 & 3.31e+05 \\ \hline
\end{tabu}
}
\caption{Estimated $L^2(\Omega_0)$ error and average number of degrees of freedom for the biquadratic adapted (AFEM) and classical (FEM) finite element approximation to the 2D trapping mode Helmholtz problem for angular frequencies $\omega=30\pi$. To estimate the error, we take the numerical approximation on a uniform mesh of width $h_K/2=1/300$ as the exact solution. }
\label{tab:err2c}
\end{table}

\begin{figure}[h]
\begin{center}
\includegraphics[width=0.45\textwidth]{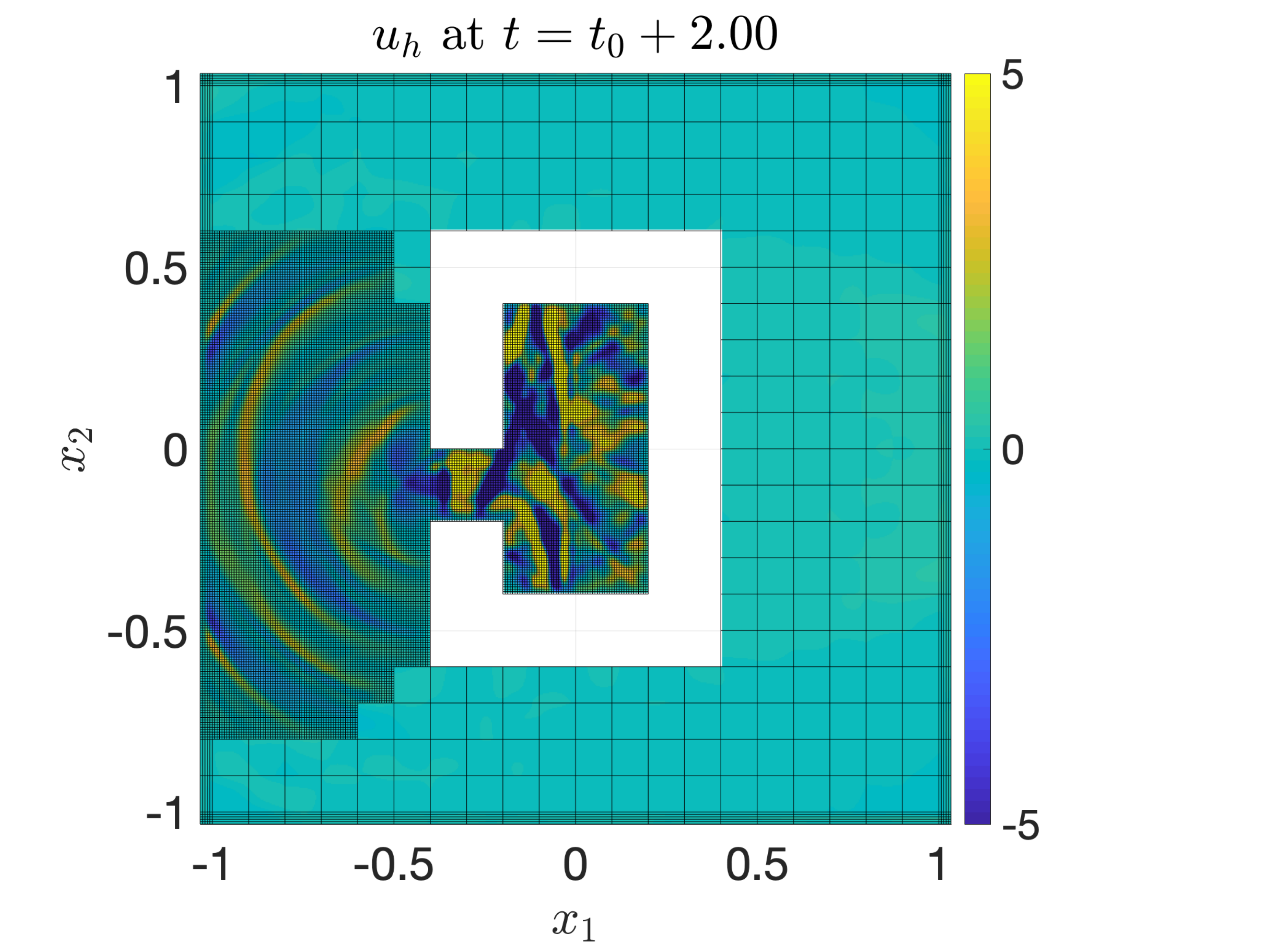} 
\includegraphics[width=0.45\textwidth]{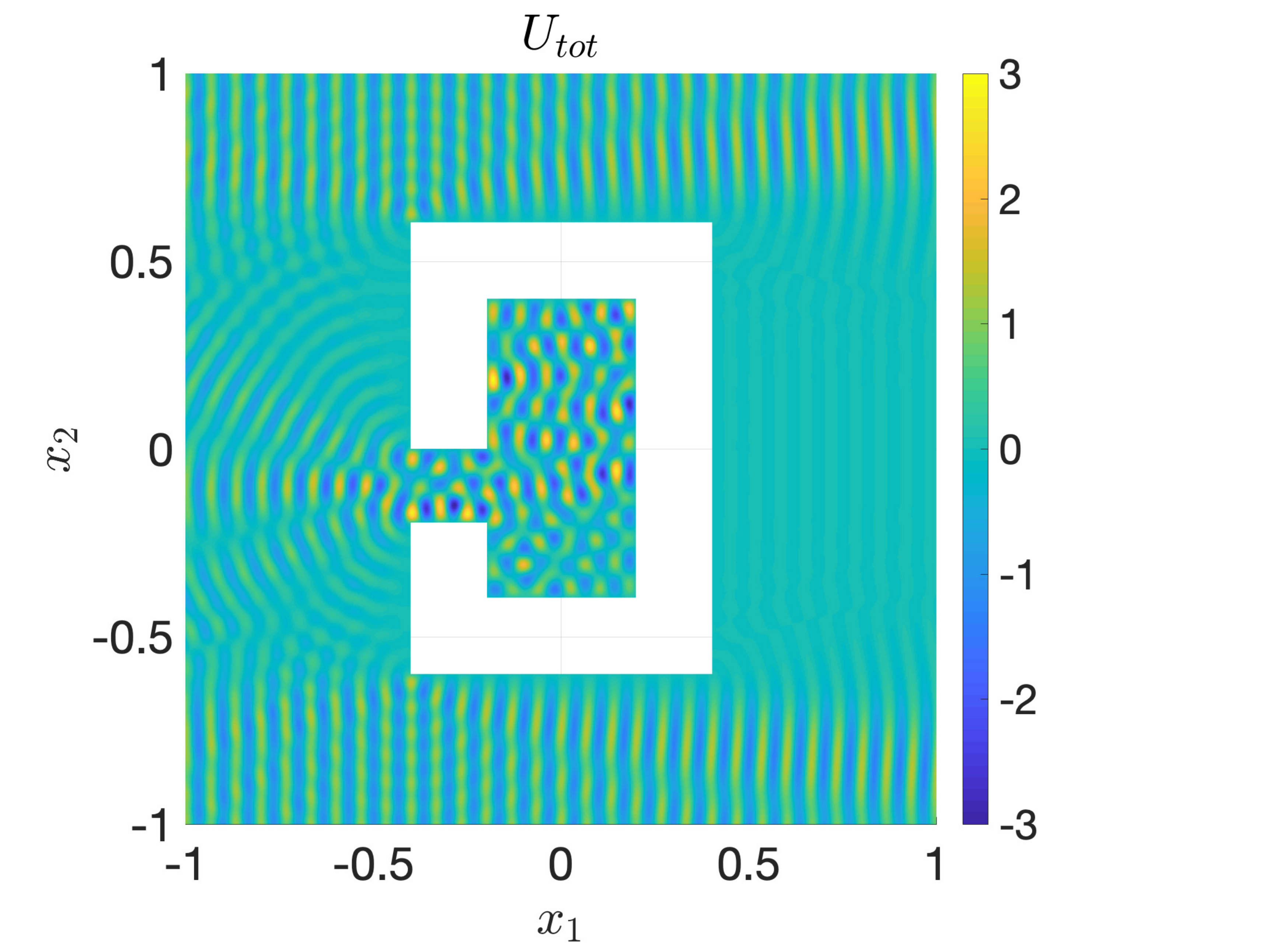} 
\end{center}
\caption{Snapshot of the wave field $u_h$ at time $T_j=t_0+2$ (left) and total time-harmonic field $U_{tot}=U_h+U_I$ (right) for the 2D trapping mode test case with $\omega=30\pi$.}
\label{fig:trappingMode}
\end{figure}

\section{Conclusion}
\label{sec:conclusion}
We considered the time-harmonic acoustic scattering problem with smoothly varying coefficients for an incoming plane wave of angular frequency $\omega$. The proposed method consists of solving the wave equation in the time domain for a single incoming plane wavelet using an adaptive mesh. The time-harmonic solution is then recovered by computing the Fourier transform in time using an adaptive algorithm that exploits the reduced number of degrees of freedom corresponding to the adapted meshes. We compared our adaptive finite element method to a standard classical finite element time domain method and show that the accuracy is comparable, whereas the average number of degrees of freedom for our adaptive method grows at a significantly smaller rate as the frequency $\omega$ increases. In particular, numerical examples indicate that the average number of degrees of freedom for the adaptive finite element method scales almost like $\cO(\omega^{d-1})$, with $d$ the number of dimensions in space, instead of $\cO(\omega^d)$. Numerical examples also demonstrate that our method can be extended to include external source terms and sound-soft scatterers. The method, however, provides only a limited advantage in the presence of trapping modes.

\appendix

\section{Limiting amplitude principle}
\label{sec:appFT}

\noindent
Let $U=U(\vx)$ be the solution to the Helmholtz problem
\begin{subequations}  
\label{eq:Helmholtz0}
\begin{align}
-\omega^2 U - \beta^{-1}\nabla\cdot(\alpha\nabla U) &= F &&\text{in }\mR^d, \\
[\text{far field radiation condition on }U], & &&
\end{align}
\end{subequations}
and let $u=u(\vx,t)$ be the
solution to the wave problem in the time domain 
\begin{subequations}  
\label{eq:wave0}
\begin{align}
\partial_t^2 u - \beta^{-1}\nabla\cdot(\alpha\nabla u) &= 
                                             f            
  &&\text{in }\mR^d\times(t_0,\infty), \\
[\text{zero initial conditions on }u\ \text{at }t=t_0], & &&
\end{align}
\end{subequations}%
with spatial parameters $\alpha=\alpha(\vx)$ and $\beta=\beta(\vx)$,
source terms $ f=f 
(\vx,t)$ and $F=F(\vx)$, and frequency $\omega>0$.

We assume that 
$\alpha(\vx)\geq\alpha_{\min}$ and $\beta(\vx)\geq\beta_{\min}$ for $\vx\in\mR^d$,
and that 
$\alpha(\vx)\equiv\alpha_0$ and $\beta(\vx)\equiv\beta_0$
for $\vx\in\mR^d\backslash\Omega_{in}$, where $\Omega_{in}$ is a bounded domain
and $\alpha_{\min}$, $\beta_{\min}$, $\alpha_0$, and $\beta_0$ are positive
constants. We also assume that $ f 
(\cdot,t)$ and $F$ 
are supported within $\Omega_{in}$. 
 
Let $\cU$ be a Hilbert space on a bounded domain $\Omega$ with
  $\Omega\supset\Omega_{in}$.
  The limiting amplitude principle states that, if $f$ 
  is of the form
  $f 
  (\vx,t)=F(\vx)e^{-\im\omega t}$, then $u(\cdot,t)$ converges in $\cU$ to $Ue^{-\im\omega t}$ as $t$ tends to infinity. In particular, we can define the limiting amplitude principle as follows:
\begin{dfn}[limiting amplitude principle] 
Let $u$ be the solution to the wave equation given in
\eqref{eq:wave0}, with $f 
(\vx,t):=F(\vx)e^{-\im\omega t}$ for all $t>T$ for some $T>t_0$, and let $U$ be the solution to the Helmholtz equation given in \eqref{eq:Helmholtz0}. The limiting amplitude principle states that
\begin{equation}\label{eq:LAP}
\lim_{t\rightarrow\infty}\left\| u\left(\cdot,t\right)-U\left(\cdot\right)e^{-i\omega t}\right \|_{\cU}=0.
\end{equation}
\end{dfn}

The following is known about the validity of the limiting amplitude
principle, with ${\rm supp}(F)\subset\Omega\subset \mR^d$.
\begin{itemize}
\item For $d=3$, \eqref{eq:LAP} was derived
in~\cite{Tamura} for $F\in L^2(\Omega)$, $\alpha\in\cC^2(\mR^3)$, $\beta\in\cC^1(\mR^3)$, $\cU=L^2(\Omega)$.
\item For $d\geq 2$, 
it follows from \cite[Ch.~2]{eidus1969} that \eqref{eq:LAP} holds true for $F\in L^2(\Omega)$, $\alpha\in\cC^2(\mR^d)$, $\beta\equiv \beta_0$, $\cU=H^1(\Omega)$.
\item For $d=1$, the form~\eqref{eq:LAP}  of the limiting
  amplitude principle is not valid~\cite[Sect.~3]{eidus1962en}; a modified form is
currently under investigation and will be presented in a forthcoming paper.
\end{itemize}

Whenever the limiting amplitude principle is valid for $\cU=L^2(\Omega)$, we have the following result.

\begin{lem}
\label{lem:FT}
Let $u$ be the solution to \eqref{eq:wave0} with a source term $f$
that has compact support in space and time. Extend $u$ and $f$ by zero
to $\mR^d\times(-\infty,t_0)$ and define, for any frequency
$\omega>0$, $\tU_\omega:=\Ft[u](\cdot,-\omega)$ and
$\tF_{\omega}:=\Ft[f](\cdot,-\omega)$, where $\Ft$ denotes the Fourier
transform with respect to time. If the limiting amplitude principle is
valid for $\cU=L^2(\Omega)$, then, for any bounded domain
$\Omega\subset \mR^d$ with ${\rm supp}(f(\cdot,t))\subset\Omega$,
we have that $\tU_{\omega}$ is the solution to the Helmholtz equation given in \eqref{eq:Helmholtz0} with source term $F=\tF_{\omega}$.
\end{lem}

\begin{proof}
Fix $\omega>0$ and let $U$ be the solution to \eqref{eq:Helmholtz0} with source term $F=\tF_{\omega}$.
We need to show that $\tU_{\omega}=U$. To do so, 
define $G(t):= H(t)e^{-\im\omega t}$, where $H(t)$ denotes the Heaviside step function ($H(t)=1$ for $t\geq 0$ and $H(t)=0$ for $t<0$). Also, let $*_t$ denote the convolution operator with respect to time. If we apply $G \; *_t$ to \eqref{eq:wave0}, we obtain
\begin{subequations}
\label{eq:waveG}
\begin{align}
\partial_t^2 (G*_t u) - \beta^{-1}\nabla\cdot(\alpha\nabla (G*_t u)) &= (G*_t f) &&\text{in }\mR^d\times(t_0,\infty), \\
[\text{zero initial conditions on }(G*_t u) \text{ at }t=t_0]. & &&
\end{align}
\end{subequations}
Since $f$ has finite support in time, we have that $f(\cdot,t)=0$ for
all $t>T$ for some $T>t_0$. Therefore, $(G*_t f)=e^{-\im\omega t}\tF_{\omega}$ for $t>T$. It then follows from the limiting
amplitude principle 
that
$(G*_t u)(\cdot,t)$ converges to $e^{-\im\omega t}U$ as
$t\rightarrow\infty$ in $\cU=L^2(\Omega)$. In other words, $\lim_{t\rightarrow\infty} e^{\im\omega t}(G*_t u)(\cdot,t) = U$. By definition of $\tU_{\omega}$, we also have that $\lim_{t\rightarrow\infty} e^{\im\omega t}(G*_t u) = \tU_{\omega}$ and hence, $\tU_{\omega}=U$.
\end{proof}

\bibliographystyle{abbrv}
\bibliography{Helmholz}

\end{document}